\documentclass{amsart}

\usepackage{all2021}
\usepackage{xypic}

\usepackage[margin=1.2in]{geometry}
\usepackage[dvipsnames]{xcolor}
\hypersetup{
     colorlinks = true,
     linkcolor = Periwinkle,
     filecolor = Red,
     citecolor = Periwinkle,
     urlcolor = Periwinkle,
     }

\usepackage[%
    backend = biber,
    style = alphabetic,
    citestyle = alphabetic,
    backref = true,
    maxalphanames  = 99,
    maxnames = 99,
    isbn = false,
    doi = false,
    url = false,
    ]{biblatex}
\newbibmacro{string+doiurlisbn}[1]{%
    \iffieldundef{url}{%
          #1%
    }{%
      \href{\thefield{url}}{#1}%
    }%
}

\DeclareFieldFormat{title}{%
  \usebibmacro{string+doiurlisbn}{\mkbibemph{#1}}}
\DeclareFieldFormat
  [article,inbook,incollection,inproceedings,patent,thesis,unpublished]
  {title}{\usebibmacro{string+doiurlisbn}{\mkbibquote{#1\isdot}}}
\DeclareFieldFormat
  [suppbook,suppcollection,suppperiodical]
  {title}{\usebibmacro{string+doiurlisbn}{#1}}

\begin{document}

\newcommand{\Sh}{\operatorname{Sh}}
\newcommand{\LM}{\operatorname{LM}}
\renewcommand{\Vec}{\mathbf{Vec}}
\newcommand{\FI}{\mathbf{FI}}
\newcommand{\FIop}{\mathbf{FI^{op}}}
\newcommand{\bC}{\mathbf{PM}}
\newcommand{\II}{\mathcal{I}}
\newcommand{\V}{\operatorname{V}}
\newcommand{\gl}{\mathfrak{gl}}

\newcommand{\Sch}{\mathbf{Sch}}
\newcommand{\VVec}{\operatorname{V}_{\mathbf{Vec}}}

\newcommand{\ADcom}[1]{\textcolor{Green}{[#1]}}

\title{S\MakeLowercase{ym}-Noetherianity for powers of GL-varieties}


\author[Chiu]{Christopher H.~Chiu}
\address{Eindhoven University of Technology, Department of Mathematics
and Computer Science, P.O.~Box 513, 5600 MB, Eindhoven, The
Netherlands}
\email{c.h.chiu@tue.nl}

\author[Danelon]{Alessandro Danelon}
\address{Eindhoven University of Technology, Department of Mathematics
and Computer Science, P.O.~Box 513, 5600 MB, Eindhoven, The
Netherlands}
\email{a.danelon@tue.nl}

\author[Draisma]{Jan Draisma}
\address{University of Bern, Mathematical Institute, Sidlerstrasse 5,
3012 Bern, Switzerland; and
Eindhoven University of Technology, Department of Mathematics
and Computer Science, P.O.~Box 513, 5600 MB, Eindhoven, The
Netherlands}
\email{jan.draisma@unibe.ch}

\author[Eggermont]{Rob H.~Eggermont}
\address{Eindhoven University of Technology, Department of Mathematics
and Computer Science, P.O.~Box 513, 5600 MB, Eindhoven, The
Netherlands}
\email{r.h.eggermont@tue.nl}

\author[Farooq]{Azhar Farooq}
\address{Eindhoven University of Technology, Department of Mathematics
and Computer Science, P.O.~Box 513, 5600 MB, Eindhoven, The
Netherlands}
\email{a.farooq@tue.nl}

\let\thefootnote\relax
\footnotetext{\hspace*{-14pt}
\begin{minipage}{.039\textwidth}
\includegraphics[width=\textwidth]{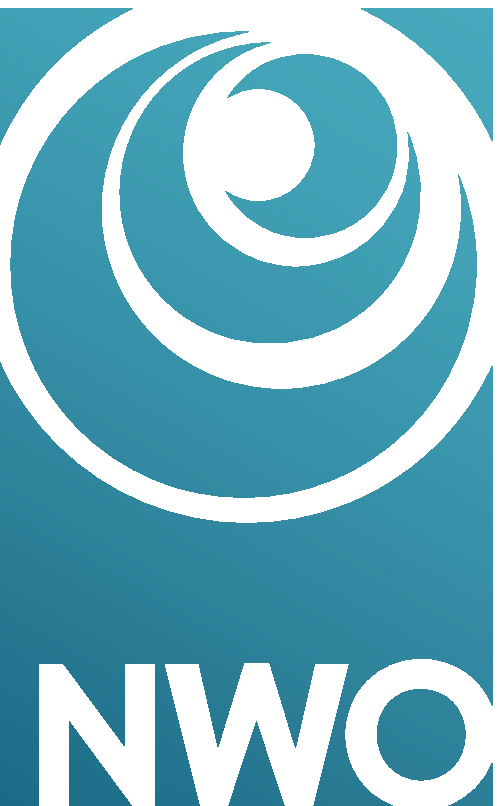}
\end{minipage}
\begin{minipage}{416pt}
CC, AD, and AF were funded by Vici grant 639.033.514 from
the Netherlands Organisation for Scientific Research (NWO).
JD was funded in part by that Vici grant and by project
grant 200021\_191981 from the Swiss National Science
Foundation. RE was supported by NWO Veni grant
016.Veni.192.113. JD thanks the Institute for Advanced
Study, where some of the work on this paper took place. 
\end{minipage}}

\maketitle

\begin{abstract}
Much recent literature concerns finiteness properties of
infinite-dimensional algebraic varieties equipped with an action of the
infinite symmetric group, or of the infinite general linear group. In
this paper, we study a common generalisation in which the product of
both groups acts on infinite-dimensional spaces, and we show
that these spaces are topologically Noetherian with respect
to this action.
\end{abstract}

\setcounter{tocdepth}{1}


\section{Introduction}

\subsection{Sym-Noetherianity and GL-Noetherianity}
\label{ssec:SymOrGL}
It has been well-established since the 1980s that if $Z$ is finite-dimensional
variety, then the topological space $Z^\NN$, equipped with the
inverse-limit topology of the Zariski topologies, has the property
that if \[ X_1 \supseteq X_2 \supseteq X_3 \supseteq \ldots \] is a
descending chain of closed subvarieties, each stable under the infinite
symmetric group $\Sym=\bigcup_n \Sym([n])$ permuting the copies of
$Z$, then $X_n=X_{n+1}$ for all $n \gg 0$. We say that $Z^\NN$ is
{\em $\Sym$-Noetherian}; see
\cite{cohen:metabelian, cohen:closure,
aschenbrenner-hillar:finitesymmetric,
hillar-sullivant:GBinfdim} for the relevant literature.

On the other hand, the third author proved that if $Z$ is a {\em
$\GL$-variety}: a (typically infinite-dimensional) affine variety
equipped with a suitable action of the infinite general linear group
$\GL=\bigcup_n \GL_n$---see below for precise definitions---then
$Z$ is topologically $\GL$-Noetherian. See \cite{draisma} for
Noetherianity, and see \cite{bik-draisma-eggermont-snowden, bik-draisma-eggermont-snowden:uniformity} for the structure theory
of $\GL$-varieties.

\subsection{Our result: Sym$\times$GL-Noetherianity}
\label{ssec:SymAndGL}
Given a $\GL$-variety $Z$, the group $\Sym \times \GL$ acts naturally
on $Z^\NN$, and our main goal in this paper is to prove the following
theorem.

\begin{thm}[Main Theorem]\label{thm:main}
Let $Z$ be a $\GL$-variety over a field of characteristic zero. Then
$Z^\NN$ is topologically $\Sym \times \GL$-Noetherian. In other words,
every descending chain
\[ X_1 \supseteq X_2 \supseteq \ldots \]
of $\Sym \times \GL$-stable closed subvarieties of $Z^\NN$ stabilises
eventually. Equivalently, any $\Sym \times \GL$-stable closed subvariety
of $Z^\NN$ is defined by finitely many $\Sym \times \GL$-orbits of
polynomial equations.
\end{thm}

Below we give two examples of $\Sym \times \GL$-varieties; these illustrate that even when $Z$ is a rather simple $\GL$-variety, $Z^\NN$ will have many $\Sym \times \GL$-stable closed subvarieties.

\begin{ex} \label{ex:Matroid}
    Consider the space of $\NN \times \NN$-matrices where $\Sym$ permutes the rows, and $\GL$ acts simultaneously on all rows.
    We can think about this space as $Z^\NN$, where $Z$ is the space $\AA^\NN$ with the obvious $\GL$-action. We write $x_{i,j}\ (i,j \in \NN)$ for the coordinates on this space.
    Let $X$ be a $\Sym\times\GL$-stable proper closed subvariety of this space.
    Let $f$ be a nonzero polynomial vanishing identically on $X$ involving only the $x_{ij}$ with $1 \leq i,j \leq n$, chosen such that $n$ is minimal among all defining equations of $X$.
    We claim that $X$ is contained in the variety of matrices with rank at most $n-1$.
    Indeed, suppose that a matrix $A$ in $X$ has rank at least $n$. Then by basic linear algebra the $\Sym \times \GL$-orbit of $A$ projects dominantly in the affine space $\AA^{n \times n}$ corresponding to the upper left $n \times n$-block. 
    This implies that $f$ is the zero polynomial; a contradiction. 
    
    Also, by the minimality of $n$, there must exist a matrix  in $X$ whose rank is $n-1$. However, it is not easy to completely classify the $\Sym \times \GL$-stable closed subvarieties of $Z^\NN$ containing a matrix of rank $n-1$ and no matrices of rank $n$. For instance, 
    fix any matroid $M$ of rank $n-1$ on the ground set $[m]\coloneqq \{1,\ldots,m\}$ and let $R \subseteq \AA^{m \times (n-1)}$ be the variety defined by the determinants of the $(n-1) \times (n-1)$-submatrices whose rows correspond to non-bases of $M$. Regard $R$ as a subvariety of $\AA^{\NN \times \NN}$ by extending with zeros and set $X_M \coloneqq \ol{(\Sym \times \GL)R}$. This $\Sym\times \GL$-stable subvariety of $\AA^{\NN \times \NN}$ is the common zero set of two classes of polynomials: all monomials containing variables from at least $m+1$ distinct rows, and the $\Sym$-orbits of all products of the form
    \[ \prod_{\pi \in \Sym([m])} \det(x[\pi(I_\pi),\;J_\pi])\]
    where each $I_\pi \subseteq [m]$ is an $(n-1)$-element set that is not a basis of $M$, each $J_\pi$ is an arbitrary $(n-1)$-element subset of $\NN$, and $x[\pi(I_\pi),J_\pi]$ stands for the matrix of variables $x_{ij} $ with $i \in \pi(I_\pi)$ and $j \in J_\pi$. 
    
Now suppose that $M,M'$ are loopless matroids on ground sets $[m],[m']$, both realisable over the algebraic closure of the ground field. We then claim that $X_M=X_M$ holds (if and) only if $M,M'$ are isomorphic. Indeed, if $X_M=X_{M'}$, then let $p \in \AA^{m \times (n-1)} \subseteq \AA^{\NN \times \NN}$ realise $M$, so that $p \in X_{M}=X_{M'}$. This means that $p$ satisfies all equations for $X_{M'}$. Since $M$ is loopless, all rows of $p$ are nonzero, and the monomial equations for $X_{M'}$ imply that $m' \geq m$. The converse follows by taking a realisation of $M'$. That the determinantal equations for $X_{M'}$ vanish on $p$ imply that, after a permutation, all non-bases of $M'$ are also non-bases of $M$. Again, the converse holds by taking a realisation of $M'$. Hence $M$ and $M'$ are isomorphic. 
    
    We conclude that the considerable combinatorial complexity of the class of realisable matroids is contained in the classification problem for $\Sym \times \GL$-subvarieties of $Z^\NN$.    
\end{ex}


\begin{re}
    Already the classification of $\Sym$-stable closed subvarieties of $(A^1)^\NN$ is nontrivial \cite{nagpal-snowden:affinesymmetric}, so it is not so surprising that also the classification of $\Sym \times \GL$-stable closed subvarieties of $Z^\NN$ in Example~\ref{ex:Matroid} is difficult.
\end{re}

\begin{ex}
    Let $Z$ be the space of symmetric $\NN \times \NN$ matrices, acted upon by $\GL$ via $(g,A) \mapsto gAg^T$. It is not hard to classify the $\GL$-stable closed subvarieties of $Z$: they are the empty set and the varieties of matrices whose rank is bounded by some $k \in \{0,\ldots,\infty\}$. 

    Now let $X$ be a $\Sym\times\GL$-stable proper closed  subvariety of $Z^\NN$, and let $n$ be minimal such that there exists a nonzero polynomial that vanishes on $X$ and involves only coordinates on the first $n$ copies of $Z$. Then it follows from \cite[Proposition 3.3]{eggermont:finitenessproperties} that $X$ is contained in the variety $X_{n,r}$ of $\NN$-tuples in which every $n$-tuple has a nontrivial linear combination whose rank is at most some integer $r$. However, completely classifying all $\Sym\times \GL$-stable closed subvarieties of $X_{n,r}$ seems completely out of reach.
\end{ex}

\subsection{A generalisation: Sym$^\textit{\textbf{k}}\times$GL-Noetherianity}
We prove the Main Theorem by
establishing first the following more general result.

\begin{thm} \label{thm:Main2}
Let $Z_1,\ldots,Z_k$ be $\GL$-varieties over a field of characteristic
zero. Then the variety $Z_1^\NN \times \cdots \times Z_k^\NN$ is $\Sym^k \times \GL$-Noetherian.
\end{thm}

Here there is one copy of $\GL$ that acts diagonally,
and there are $k$ copies of $\Sym$ that act on separate copies of
$\NN$.
We believe it is impossible to prove the Main Theorem without considering multiple copies of $\Sym$.
Indeed, covering a proper closed $\Sym \times \GL$-stable subvariety of $Z^\NN$ requires partitioning $\NN$ into finitely many parts such that the points in $Z$ labelled by the indices in one same part behave in a similar fashion.
The following example illustrates this point.

\begin{ex}\label{ex:SymSkew}
Let $Z$ be the space of $\NN \times \NN$-matrices over a field of
characteristic zero, equipped with the
$\GL$-action given by $(g,A) \mapsto gAg^T$. Let $X$ be the closed $\Sym \times \GL$-stable subvariety of $Z^\NN$ consisting of all infinite
matrix tuples $(A_1,A_2,\ldots)$ such that each $A_i$ is either symmetric
or skew-symmetric.  It is easy to see that $X$ is defined by the $\Sym \times \GL$-orbit of the equation $(x_{112}+x_{121})(x_{112}-x_{121})$,
where $x_{ijk}$ is the $(j,k)$-entry of the $i$th matrix. We will
see that the $\Sym \times \GL$-Noetherianity of $X$ follows from the
$\Sym^2 \times \GL$-Noetherianity of the \textit{``smaller''} variety $Z_1^{\NN} \times Z_2^{\NN}$,
where $Z_1 \subseteq Z$ is the $\GL$-subvariety of symmetric matrices,
and $Z_2 \subseteq Z$ is the $\GL$-subvariety of skew-symmetric matrices.
Here the term ``smaller'' refers to the fact that both $Z_1$ and $Z_2$ are quotients of $Z$.
The exact meaning of smaller varieties is given in Section~\ref{sssec:orderVecembedded}.
\end{ex}

\subsection{Relation to existing literature}

The Main Theorem generalises the results mentioned in Section~\ref{ssec:SymOrGL}:
taking for $Z$ a finite-dimensional affine variety with trivial $\GL$-action,
one recovers the $\Sym$-Noetherianity of $Z^\NN$; and on the other
hand, if $Z$ is a $\GL$-variety, then considering
chains $X_1 \supseteq X_2 \supseteq \ldots$ in which each $X_i$ is of the form
$Z_i^\NN$ with $Z_i \subseteq Z$ a $\GL$-subvariety, one recovers the
$\GL$-Noetherianity of $Z$.

The proof of the Main Theorem will reflect these two special cases.
We will use the proof method from \cite{draisma} for the $\GL$-Noetherianity of $Z$,
and similarly, we will use methods for $\Sym$-varieties from
\cite{draisma-eggermont-farooq:components}. In fact, we do not
explicitly use Higman's lemma in our proofs as is classically done
\cite{aschenbrenner-hillar:finitesymmetric,hillar-sullivant:GBinfdim,draisma:combinatorialAG}, and {\em en passant} we give a new
proof of the $\Sym$-Noetherianity of $Z^\NN$ for a finite-dimensional
variety $Z$. However, our proof only yields a {\em set-theoretic}
Noetherianity result, while in the pure $\Sym$-setting (much) stronger
results are known: increasing chains of $\Sym$-stable ideals in the coordinate ring of $Z^\NN$ with $Z$ a finite-dimensional variety stabilise
\cite{cohen:metabelian,cohen:closure,aschenbrenner-hillar:finitesymmetric,hillar-sullivant:GBinfdim}, and even
finitely generated modules over such rings with a compatible $\Sym$-action
are Noetherian \cite{Nagel17b}.  In the pure $\GL$-setting, however, such
stronger Noetherianity results are known only for very few classes of $\GL$-varieties: over a field of characteristic zero  ring-theoretic Noetherianity holds for a direct
sum of copies of the first symmetric power $S^1$ \cite{sam-snowden:GLmodulesoverpolyringsininfvars, sam-snowden:GLmodulesoverpolyringsininfvarsII}, for the second symmetric power $S^2$, for $\bigwedge^2$ \cite{nagpal-sam-snowden}, for $S^1 \oplus S^2$
and for $S^1 \oplus \bigwedge^2$ \cite{Sam22}.

Partitions of $\NN$ into finitely many subsets also feature in the
classification of symmetric subvarieties of infinite affine space
$(\AA^1)^\NN$ \cite{nagpal-snowden:affinesymmetric}, and while our proofs do not logically
depend on this classification, that paper did serve as an inspiration.

\subsection{Organisation of this paper}

This paper is organised as follows.
In Sections~\ref{ssec:Vecvar} and~\ref{ssec:FIxVecvars} we
introduce polynomial functors and affine varieties over the categories $\Vec$, $\FIop \times \Vec$ and $(\FIop)^k \times \Vec$.
This language happens to be more convenient than a purely infinite-dimensional approach, as shown in  Remark~\ref{re:Richer}.
In Section~\ref{ssec:PM} we introduce the category
$\bC$ with morphisms between $(\FIop)^k \times \Vec$-varieties, in which, for
the reasons explained in Example~\ref{ex:SymSkew} and above it, $k$ varies.
In Section~\ref{ssec:producttype} we describe $(\FIop)^k \times \Vec$-varieties of {\em product type}. The simplest ones among these are of the form
\[ Z:
(V;S_1,\ldots,S_k) \mapsto \prod_{i=1}^k Z_i(V)^{S_i},\]
which are the ones of interest in our Theorem~\ref{thm:main} and Theorem~\ref{thm:Main2}.
Reformulations of our Main Theorem and its generalisation Theorem~\ref{thm:Main2} in this language are in Remark~\ref{re:reformulation}.

Our proofs rely on induction on the ``complexity'' of product-type $(\FIop)^k \times \Vec$-varieties.
The several well-founded orders used in this induction are the topic of Section~\ref{ss:complexity}, which builds on $\FI$-techniques developed in Sections~\ref{ssec:shifting} and~\ref{ssec:leadingmonideal}.
We introduce orders on:
\begin{enumerate}
    \item polynomial functors (Section~\ref{sssec:orderpolfun}),
    \item $\Vec$-varieties with a specified closed embedding in $B \times Q$ where $B$ is a finite-dimensional algebraic variety and $Q$ is a suitable polynomial functor (Section~\ref{sssec:orderVecembedded}),
    \item $(\FIop)^k \times \Vec$-varieties of product type in the category $\bC$ (Section~\ref{sssec:ProductTypeOrder}),
    \item \label{closedsubvar} closed subvarieties of $(\FIop)^k \times \Vec$-varieties of product type (Section~\ref{sssec:orderclosedproducttype}).
\end{enumerate}
Then in Section~\ref{sec:Core} we formulate and
prove the Parameterisation Theorem,  Theorem~\ref{thm:parameterisation}, the core technical result of this
paper.
The statement roughly says that if $X$ is a proper closed $(\FIop)^k \times \Vec$-subvariety of a variety $Z$ of  product type, then $X$ is covered by finitely many morphisms
in $\bC$ from $(\FIop)^l \times \Vec$-varieties of product form that are smaller than $Z$ in the sense of Section~\ref{sssec:ProductTypeOrder}.
We prove this theorem by induction on closed subvarieties mentioned in (\ref{closedsubvar}).
The description of these smaller $(\FIop)^l \times \Vec$-varieties of product type relies on Proposition~\ref{prop:Smaller}.
This proposition allows to partition points according to their common behaviour with respect to a specific defining equation, similarly to what happens in \cite{draisma}.
Indeed, our Lemma~\ref{lm:Dra19} is proven as an iteration of the argument for the Embedding Theorem in \cite{bik-draisma-eggermont-snowden}, which in turn uses a technique developed in \cite{draisma}.
Essential for applying Proposition~\ref{prop:Smaller} is the operation of shifting over a tuple of finite sets, described in Section~\ref{ssec:shifting}.
In the final Section~\ref{sec:Proof} we use all the above to prove that all $(\FIop)^k \times \Vec$-varieties of product type are Noetherian via an induction on their order of Section~\ref{sssec:ProductTypeOrder}.
Theorem~\ref{thm:Main2} and the Main Theorem follow as corollaries.

\subsection{Notation and conventions}
\begin{itemize}
    \item For a nonnegative integer $k$, we set $[k]\coloneqq \{1,
    \dots, k\}$; so in particular $[0]=\emptyset$.
    \item Let $S$ be a finite set. We denote by $|S|$ the cardinality of $S$.
    \item Throughout this paper, we work over a field $K$ of
    characteristic zero.
    \item $\Sym$ denotes the infinite symmetric group.
    It is defined as the direct limit over $\Sym(n)$, the symmetric group on the set $[n]$, with the obvious inclusion maps.
    \item $\GL$ denotes the infinite general linear group.
    It is defined as the direct limit of $\GL_n$, the general linear group on $K^n$, with inclusion maps $\GL_n \to \GL_{n+1}$ given by 
    \[ g \mapsto
    \left(\begin{array}{c|c}
        g & 0 \\
       \hline 0 & 1
    \end{array}\right).
    \]
    \item The category of schemes over $K$ is denoted by $\Sch_K$. A product $X\times Y$ of two schemes will always mean a product in this category.
    \item A \emph{variety} $X$ here is a reduced affine scheme of finite type over $K$. By $K[X]$ we denote its coordinate ring, so $X = \Spec K[X]$. If $Y$ is a subvariety of $X$, then we write $\II(Y) \subseteq K[X]$ for the (radical) ideal of functions on $X$ vanishing on $Y$.
    \item If $f\in K[X]$ then we write $X[1/f] \coloneqq \Spec(K[X]_f)$.
    \item Let $\phi: X \to Y$ be a morphism of varieties.
    We denote by $\phi^\#: K[Y] \to K[X]$ the induced morphism on coordinate rings.
    \item By a point $x$ of a variety $X$ we always mean a closed point of $X$, i.e.\ an element of $X(\overline{K})$.
\end{itemize}

\section{The categories of $(\FIop)^k \times \Vec$-varieties}\label{sec:Setup}

\subsection{$\Vec$-varieties}\label{ssec:Vecvar}
Let $K$ be a field of characteristic zero, and let $\Vec$ be the category
of finite-dimensional vector spaces over $K$ with $K$-linear morphisms.
We will be working with $\Vec$-varieties, a functorial finite-dimensional counterpart of $\GL$-varieties.
Below, we quickly recap the theory of polynomial functors: definitions, relevant properties; and we define the notion of $\Vec$-variety.
See Remark~\ref{re:GLVar} for the connection with $\GL$-varieties.

\begin{de}
A {\em polynomial functor} is a functor $P:\Vec \to \Vec$ such that for each
$U,V \in \Vec$ the map $P:\Hom_\Vec(U,V) \to \Hom_\Vec(P(U),P(V))$ is
polynomial, and such that the degree of this polynomial map is upper-bounded
independently of $U,V$. The minimal such bound is called the \emph{degree}
of $P$.
\end{de}

We will also regard a polynomial functor $P$ as a functor $\Vec \to \Sch_K$ by composing with the embedding $\Vec \to \Sch_K$ given by $V \mapsto \Spec \left (\Sym_K (V^*)\right)$, the spectrum of the symmetric algebra on the dual space $V^*$ of $V$.
Every polynomial functor $P$ equals $P_0 \oplus \cdots
\oplus P_d$, where $d$ is the degree of $P$ and $P_i$ is defined as
\[ P_i(V)\coloneqq \{v \in P(V) \mid \forall t \in K: P(t\id_V)v=t^i v\}. \]
Considering $P$ as a functor $\Vec \to \Sch_K$ we have $P(V) = P_0(V) \times \ldots \times P_d(V)$. 
We note that $P_0$ is
a constant polynomial functor, which assigns a fixed vector space $P(0)
\in \Vec$ to all $V \in \Vec$ and the identity map to each linear map. We
call $P$ {\em pure} if $P_0=\{0\}$.

Let $X,Y: \Vec \to \Sch_K$ be functors. A \emph{closed immersion} $\iota: X \to Y$ is a natural transformation such that $\iota(V): X(V) \to Y(V)$ is a closed immersion for all $V\in\Vec$. In particular, $X$ is then a subfunctor of $Y$.

\begin{de}
An \emph{affine $\Vec$-scheme} is a functor $X: \Vec \to \Sch_K$ that admits a closed immersion $X \to P$ with $P: \Vec \to \Sch_K$ a polynomial functor. A \emph{$\Vec$-variety} is an affine $\Vec$-scheme $X$ such that $X(V)$ is reduced for all $V\in \Vec$. The \emph{category of affine $\Vec$-schemes} is the full subcategory of the functor category $\Sch_K^{\Vec}$ whose objects are affine $\Vec$-schemes.
\end{de}

Spelled out explicitly, a $\Vec$-variety $X$ can be
described by the data of a
polynomial functor $P$ and a subvariety $X(V) \subseteq P(V)$ for each $V\in \Vec$ such that, for each $\phi \in
\Hom_{\Vec}(U,V)$, the linear map $P(\phi)$ maps $X(U)$ into $X(V)$. A morphism of $\Vec$-varieties $\tau:X\to Y$ consists of
a morphism of varieties $\tau(V): X(V) \to Y(V)$ for each $V\in\Vec$ such that, for each $\phi \in \Hom_{\Vec}(U,V)$, we have $\tau(V) \circ X(\phi)=Y(\phi) \circ \tau(U)$.

\begin{re}\label{re:VecProducts}
The subcategory of $\Vec$-varieties is closed under taking closed immersions and finite products. To see the latter, note that the product of $X,Y: \Vec \to \Sch_K$ in $\Sch_K^{\Vec}$ is given by $V \mapsto X(V) \times Y(V)$; and furthermore, given closed immersions $X \hookrightarrow P$, $Y \hookrightarrow Q$ the assignment
\[X(V)\times Y(V) \to P(V)\times Q(V)\]
defines a closed immersion of the product $X\times Y$ into the polynomial functor $P \oplus Q$.
\end{re}

\begin{lm}
The category of affine $\Vec$-schemes admits fibre products.
\end{lm}

\begin{proof}
First note that for morphisms of affine $\Vec$-schemes $X \to Y$, $Z
\to Y$ the fibre product $X\times_Y Z$ of $X$ and $Z$ over $Y$ exists in the functor category $\Sch_K^{\Vec}$ and is given by
\[(X\times_Y Z)(V)\coloneqq X(V) \times_{Y(V)} Z(V).\]
Moreover, since $Y(V)$ is affine (or more generally since $Y(V)$ is separated, see \cite[\href{https://stacks.math.columbia.edu/tag/01KR}{Tag 01KR}]{stacks-project}) the natural morphism
$X(V) \times_{Y(V)} Z(V) \to X(V) \times Z(V)$
is a closed immersion. The statement then follows by Remark~\ref{re:VecProducts}.
\end{proof}

The main result of \cite{draisma} says that $\Vec$-varieties are topologically Noetherian.

\begin{thm}[{\cite[Theorem 1]{draisma}}]\label{thm:Noetherian}
Let $X$ be a $\Vec$-variety. Then every descending chain of $\Vec$-subvarieties
\[X = X_0 \supseteq X_1 \supseteq X_2 \supseteq \ldots \]
stabilizes, that is, there exists $N\geq0$ such that for each $n\geq N$ we have $X_n = X_{n+1}$.
\end{thm}

\begin{re} \label{re:GLVar}
If $X$ is a $\Vec$-variety, then $X_\infty \coloneqq \lim_{\ot n} X(K^n)$ is a
$\GL$-variety in the sense of \cite{bik-draisma-eggermont-snowden}. This yields an equivalence
of categories between $\Vec$-varieties and $\GL$-varieties. Most of
our reasoning will be in the former terminology, but could be rephrased in
the latter.
\end{re}

\subsection{$(\FIop)^k \times \Vec$-varieties}\label{ssec:FIxVecvars}
Let $\FI$ be the category of finite sets with injections.

\begin{de}
Let $k \in \ZZ_{\geq0}$. An {\em $(\FIop)^k \times \Vec$-variety}
is a covariant functor $X$ from $(\FIop)^k$ to the category of
$\Vec$-varieties.
\end{de}

Explicitly, an $(\FIop)^k \times \Vec$-variety is given by the following data: for any $k$-tuple $(S_1,\ldots,S_k)$
we have a $\Vec$-variety $X(S_1,\ldots,S_k)$, and for any $k$-tuple
of injective maps $\iota=(\iota_1:S_1 \to T_1,\ldots,\iota_k:S_k \to
T_k)$, we have a corresponding morphism $X(\iota):X(T_1,\ldots,T_k)
\to X(S_1,\ldots,S_k)$ of $\Vec$-varieties and the usual
requirements that $X(\tau \circ \iota)=X(\iota) \circ X(\tau)$ and
$X(\id_{S_1},\ldots,\id_{S_k})=\id_{X(S_1,\ldots,S_k)}$.

Again, there are natural notions of morphism and closed immersion of $(\FIop)^k
\times \Vec$-varieties, and we call an
$(\FIop)^k \times \Vec$-variety Noetherian if every descending chain of closed
$(\FIop)^k \times \Vec$-subvarieties stabilises.

\begin{re} \label{re:Product}
In particular, any contravariant functor from $\FI$ to finite-dimensional
affine varieties, i.e., an $\FIop$-variety, is trivially an $\FIop \times \Vec$-variety. In this generality, $\FIop$-varieties are certainly not
Noetherian: see \cite[Example 3.8]{hillar-sullivant:GBinfdim}.

However, we will be largely concerned with $(\FIop)^k \times \Vec$-varieties
defined as follows. Let $Z_1,\ldots,Z_k$ be $\Vec$-varieties, define
\begin{equation}\label{eq:Product} X(S_1,\ldots,S_k)\coloneqq
Z_1^{S_1} \times \cdots \times Z_k^{S_k} \end{equation}
and for $\iota=(\iota_1,\ldots,\iota_k):(S_1,\ldots,S_k) \to (T_1,\ldots,T_k)$
define $X(\iota)$ as the product of the natural projections $Z^{T_i} \to
Z^{S_i}$ associated to $\iota_i$. We will prove that $(\FIop)^k \times \Vec$-varieties of this form are, indeed, Noetherian.
\end{re}

Note that we may also regard a $(\FIop)^k \times \Vec$-variety as a functor $(\FIop)^k \times \Vec \to \Sch_K$. For fixed $k$, the $(\FIop)^k \times \Vec$-varieties thus form a category by considering it as the full subcategory in the corresponding functor category.

\begin{re} \label{re:Richer}
If $X$ is an $(\FIop)^k \times \Vec$-variety, then the group $\Sym^k
\times \GL$ acts on the inverse limit
\begin{displaymath}
\varprojlim_{ n_1,\ldots,n_k,n} X([n_1],\ldots,[n_k])(K^n).
\end{displaymath}
This gives a functor from $(\FIop)^k \times \Vec$-varieties
to (infinite-dimensional) schemes equipped with a $\Sym^k \times \GL$-action.  Unlike in Remark~\ref{re:GLVar},
this is not quite an equivalence of categories (even under reasonable
restrictions on the $\Sym^k \times \GL$-action).  For example,
$X([n_1],\ldots,[n_k])$ could be empty for large $n_i$
and a fixed nontrivial $\GL$-variety for smaller $n_i$. We will consider an explicit example of this type later in Example~\ref{ex:fewer}. In that case,
the inverse limit is empty but the $(\FIop)^k \times \Vec$-variety is
not trivial. Our theorems will be formulated in the richer category of $(\FIop)^k
\times \Vec$-varieties.
\end{re}

\subsection{Partition morphisms and the category $\bC$}\label{ssec:PM}

Suppose that we are given a point $p$ in some $X(S_1,\ldots,S_k)(V)$, where $X$ is as in \eqref{eq:Product}. Then the components of $p$ labelled by one of the finite sets $S_i$ may exhibit different behaviours, which prompts us to further partition $S_i$ into subsets labelling components where the behaviour is similar. See Example~\ref{ex:SymSkew} for an instance of this phenomenon. In that case, $p$ will be in the image of some partition morphism; for Example~\ref{ex:SymSkew} this is further explained in Example~\ref{ex:partmorf}. Partition morphisms are defined below,  after another motivating example.   

\begin{ex}
We revisit a step in the classification of $\Sym$-invariant subvarieties of infinite affine space from \cite{nagpal-snowden:affinesymmetric}. We do so in the $\FI$-framework, where this corresponds to closed $\FIop$-subvarieties of the $\FIop$-variety $X(S)\coloneqq (\AA^1)^S=\AA^S$, where, for an injection $\pi:S \to T$, the map $X(\pi)$ is the corresponding projection $\AA^T \to \AA^S$. Let $Z$ be a proper closed $\FIop$-subvariety of $X$. By (the $\FI$-analogue of) \cite[Proposition 2.6]{nagpal-snowden:affinesymmetric}, the number of distinct coordinates of points in $Z(S)$ is bounded by some natural number $l$, independently of $S$. This means that for every $S \in \FI$, $Z(S)$ is contained in the union over all partitions of $S$ into subsets $T_1,\ldots,T_l$ of the morphism $\phi(T_1,\ldots,T_l):\AA^l \to X(S)$ that maps $(p_1,\ldots,p_l)$ to the tuple $(q_i)_{i \in S}$ with $q_i=p_j$ for the unique $j \in [l]$ with $i \in S_j$. The morphisms $\phi(T_1,\ldots,T_l)$ for varying $(T_1,\ldots,T_l) \in \FI^l$ form a partition morphism into $X$ from the constant $(\FIop)^l$-variety $Y:(T_1,\ldots,T_l) \mapsto \AA^l$, an object that is arguably simpler than $X$. In the definition that follows, we generalise this notion to the setting where $X$ is an arbitrary $(\FIop)^k \times \Vec$-variety.
\end{ex}

\begin{de}
Let $X$ be an $(\FIop)^k \times \Vec$-variety and let $Y$ be an $(\FIop)^l \times \Vec$-variety. A {\em partition morphism} $Y \to X$ consists of
the following data:
\begin{enumerate}
\item a map $\pi:[l] \to [k]$; and
\item for each $l$-tuple of finite sets $(T_1,\ldots,T_l)$ a morphism
\[
\phi(T_1,\ldots,T_l):Y(T_1,\ldots,T_l) \to
X\left(\;\bigsqcup_{j \in \pi^{-1}(1)} T_j\ ,\ldots,\bigsqcup_{j \in
\pi^{-1}(k)} T_j\right)
\]
of $\Vec$-varieties in such a manner that for any $l$-tuple $\iota=(\iota_j)_j
\in \Hom_\FI(S_j,T_j)^l$ the following diagram of $\Vec$-variety
morphisms commutes:
\[
\xymatrix{
Y(T_1,\ldots,T_l) \ar[rr]^-{\phi(T_1,\ldots,T_l)}
\ar[d]_{Y(\iota_1,\ldots,\iota_l)} && X\left(\bigsqcup_{j \in \pi^{-1}(1)}
	T_j,\ldots,\bigsqcup_{j \in \pi^{-1}(k)} T_j\right)
	\ar[d]^{X\left(\bigsqcup_{j \in \pi^{-1}(1)}
	\iota_j,\ldots,\bigsqcup_{j \in \pi^{-1}(k)} \iota_j\right)} \\
Y(S_1,\ldots,S_l) \ar[rr]_-{\phi(S_1,\ldots,S_l)} &&
X\left(\bigsqcup_{j \in \pi^{-1}(1)} S_j,\ldots,\bigsqcup_{j \in
\pi^{-1}(k)} S_j\right). }
\]
\end{enumerate}
\end{de}

\begin{re}
Note that if we take $k=l$ and $\pi=\id_{[k]}$, then a partition morphism is just a
morphism of $(\FIop)^k \times \Vec$-varieties.
\end{re}

There is a natural way to compose partition morphisms: if $(\pi,\phi)$
is a partition morphism $Y \to X$ as above and $(\rho,\psi)$ is a
partition morphism $Z \to
Y$, where $Z$ is an $(\FIop)^m \times \Vec$-variety, then $(\pi,\phi)
\circ (\rho,\psi)$ is the partition morphism given by the data $\pi \circ
\rho:[m] \to [k]$ and the morphisms
\begin{align*}
&\phi\left(\bigsqcup_{n \in \rho^{-1}(1)} R_n\ ,\ldots,\bigsqcup_{n \in
\rho^{-1}(l)} R_n \right) \
\circ \ \psi(R_1,\ldots,R_m): \\
&Z(R_1,\ldots,R_m) \to X\left(\bigsqcup_{n \in (\pi \circ
\rho)^{-1}(1)}
R_n\ ,\ldots,\bigsqcup_{n \in (\pi \circ \rho)^{-1}(k)} R_n \right).
\end{align*}
A tedious but straightforward computation shows that partition
morphisms turn the class of $(\FIop)^k \times \Vec$-varieties, with
varying $k$, into a category. We call this category $\bC$.

\begin{de} \label{de:Image}
Let $X$ be an $(\FIop)^k \times \Vec$-variety, $Y$ an $(\FIop)^l \times \Vec$-variety, and $(\pi,\phi):Y \to X$ a partition morphism. Let
$S_1,\ldots,S_k \in \FI$ and $V \in \Vec$. The (set-theoretic)
{\em image} of $(\pi,\phi)$ in $X(S_1,\ldots,S_k)(V)$ is defined
as the set of all points of the form $(X(\iota_1,\ldots,\iota_k)(V) \circ
\phi(T_1,\ldots,T_l)(V))(q)$ where $T_1,\ldots,T_l$ are finite sets,
$q$ is a point in $Y(T_1,\ldots,T_l)(V)$, and each $\iota_i$ is
a bijection from $S_i$ to $\bigsqcup_{j \in \pi^{-1}(i)} T_j$. The
partition morphism $(\pi,\phi)$ is called {\em surjective} if its image
in $X(S_1,\ldots,S_k)(V)$ equals $X(S_1,\ldots,S_k)(V)$ for all choices
of $S_1,\ldots,S_k$ and $V$.
\end{de}

\begin{re}
In the previous definition, each bijection $\iota_i$ induces
a partition of the set $S_i$. Furthermore,
if a partition morphism $(\pi,\phi)$ is surjective and for
every $i$ the $\Vec$-variety
\[X(\emptyset,\ldots,\emptyset,\{*\},\emptyset,\ldots,\emptyset),\]
where $\{*\}$ is a singleton in the $i$-th position, is nonempty, then the map $\pi$
is automatically surjective, so that $\pi$ induces a partition of $[l]$
into $k$ labelled, nonempty parts.
This is our reason for calling the morphisms in $\bC$ partition morphisms.
\end{re}

The following example rephrases Example~\ref{ex:SymSkew} in the
current terminology.

\begin{ex}\label{ex:partmorf}
Let $Z$ be the $\Vec$-variety that maps $V$ to $V \otimes V$, and let
$Z_1,Z_2$ be the closed $\Vec$-subvarieties consisting of symmetric
and skew-symmetric tensors, respectively.
Consider the $\FIop \times \Vec$-variety defined by $S \mapsto Z^S$, and for every finite set $S$ let $X(S)$ be the closed $\Vec$-subvariety given by the points $x = (x_s)_{s\in S} \in Z(V)^S$ such that each component $x_s$ is either symmetric or skew-symmetric.
Note that $X$ is a closed $\FIop\times \Vec$-subvariety.
Let $Y$ be the $(\FIop)^2 \times \Vec$-variety defined by 
\begin{displaymath}
Y(S_1,S_2) = Z_1^{S_1} \times
Z_2^{S_2}.
\end{displaymath}
We now construct a partition morphism $\phi: Y \to X$ as follows.
The map $\pi:[2] \to [1]$ is the only possible, and for every $V\in \Vec$ and $(S_1, S_2) \in \FIop^2$ the map
\begin{displaymath}
\phi(S_1,S_2)(V): Y(S_1,S_2)(V) = Z_1(V)^{S_1} \times Z_2(V)^{S_2}  \to X(S_1 \sqcup S_2)(V)
\end{displaymath}
is defined by:
\begin{displaymath}
((x_{s_1})_{s_1\in S_1},(x_{s_2})_{s_2\in S_2})   \mapsto (x_s)_{s\in S_1 \sqcup S_2}.
\end{displaymath}
Note that the partition morphism $\phi$ is surjective.
In particular, we say that $X$ \textit{is covered} by $Y$, and, as we have already hinted in Example~\ref{ex:SymSkew}, $Y$ is in some sense smaller than the assignment $S \mapsto Z^S$.
The fact that we can do this in general is the content of the {\em Parameterisation Theorem~\ref{thm:parameterisation}}.
\end{ex}

The following lemma is immediate.

\begin{lm} \label{lm:Preimage}
Let $X$ be an $(\FIop)^k \times \Vec$-variety, $X'$ a closed $(\FIop)^k \times \Vec$-subvariety of $X$, and let $(\pi,\phi)$ be a partition
morphism from an $(\FIop)^l \times \Vec$-variety $Y$ to $X$. Then
$Y'\coloneqq (\pi,\phi)^{-1}(X')$ defined by
\[
Y'(T_1,\ldots,T_l) \coloneqq \phi(T_1,\ldots,T_l)^{-1}
\left(X'\left(\;\bigsqcup_{j \in \pi^{-1}(1)}T_j\ ,\ldots, \bigsqcup_{j
\in \pi^{-1}(k)}T_j \right)\right)
\]
is a closed $(\FIop)^l \times \Vec$-subvariety of $Y$, and the data of
$\pi$ together with the restrictions of the morphisms $\phi(T_1,\ldots,T_l)$ gives
a partition morphism from $Y'$ to $X$. Moreover, if $(\pi,\phi)$ is surjective, then so is its restriction to $Y' \to X'$.
\end{lm}

The following easy proposition is crucial in our approach to the
main theorem.

\begin{prop}\label{prop:coverNoetherian}
If $(\pi,\phi)$ is a surjective partition morphism from $Y$ to $X$, and $Y$ is a Noetherian $(\FIop)^l \times \Vec$-variety, then $X$ is a Noetherian
$(\FIop)^k \times \Vec$-variety.
\end{prop}

\begin{proof}
Let $X_1 \supseteq X_2 \supseteq \ldots$ be a descending chain of closed
$(\FIop)^k \times \Vec$-subvarieties. By Lemma~\ref{lm:Preimage},
the preimages $Y_i\coloneqq (\pi,\phi)^{-1}(X_i)$ are closed $(\FIop)^l \times \Vec$-subvarieties of $Y$. Hence the chain $Y_1 \supseteq Y_2
\supseteq \ldots$ stabilises by assumption.  The surjectivity of
$(\pi,\phi)$ implies the surjectivity of its restriction to $Y_i \to X_i$. This implies that $X_i$ is uniquely determined
by $Y_i$, and hence the chain $X_1 \supseteq X_2 \supseteq \ldots$
stabilises at the same point.
\end{proof}

\subsection{Product type}\label{ssec:producttype}

We now introduce the $(\FIop)^k \times \Vec$ varieties of product
type. Essentially, these are the varieties from Remark~\ref{re:Product},
but for our proofs we will need a finer control over these
products. Therefore, we will work over a general base $\Vec$-variety $Y$,
and keep track of the ``constant parts'' $B_i$ of the $\Vec$-varieties whose
products we consider.

\begin{de}\label{de:producttype}
Let $Y$ be a $\Vec$-variety and $k,n_1,\ldots,n_k\in\ZZ_{\geq 0}$.
For each $i \in [k]$, let $B_i$ be a $\Vec$-subvariety of $Y \times \AA^{n_i}$, and $Q_i$ be a pure polynomial functor.
By construction each $\Vec$-variety $B_i \times Q_i$ has a morphism to $Y$ induced by the projection $Y \times \AA^{n_i} \to Y$.
We define the $(\FIop)^k \times \Vec$-variety $Z=[Y;B_1 \times Q_1,\ldots,B_k \times Q_k]$ via
\[Z(S_1,\ldots,S_k)\coloneqq \underbrace{(B_1 \times Q_1) \times_Y \ldots \times_Y (B_1 \times Q_1)}_{\text{cardinality-of-}S_1 \text{ times}} \times_Y (B_2 \times Q_2) \times_Y \ldots \times_Y (B_k \times Q_k),\]
where for every index $i \in [k]$ the fibre product over $Y$ of $B_i\times Q_i$ with itself is taken
$|S_i|$ times, and these copies are
labelled by the elements of $S_i$; and where the morphism $Z(T_1,\ldots,T_k) \to
Z(S_1,\ldots,S_k)$ corresponding to $\iota:S \to T$ is the
projection as in Remark~\ref{re:Product}. 
We also write the above product in a more compact notation as
\[
(B_1\times Q_1)^{S_1}_Y \times_Y \cdots \times_Y (B_k \times Q_k)^{S_k}_Y.
\]
We say that $Z$ is an $(\FIop)^k \times \Vec$-variety of
\textit{product type} (over $Y$).
\end{de}

Note that $Z(S_1,\ldots,S_k)$ is naturally a closed $\Vec$-subvariety of
\[ Y \times \prod_{i=1}^k (\AA^{n_i} \times Q_i)^{S_i},\]
where the product is over $K$. 
Moreover, if $k=0$, then by definition $Z = Y$.

When we talk of $(\FIop)^k \times \Vec$-varieties of product type, we
will always specify each $B_i$ together with its closed embedding in
$Y \times \AA^{n_i}$;
the reason being that, in the proof of the Main Theorem, we aim to argue by induction on both $Y$ and $n_i$.

\begin{re}\label{re:reformulation}
The settings of Theorem~\ref{thm:main} and Theorem~\ref{thm:Main2} can
be rephrased in our current terminology as follows.
Consider $\Vec$-varieties $Z_1, \dots, Z_k$. Then for every $i \in
[k]$ there exist $n_i \in \ZZ_{\geq 0}$, a closed subvariety $A_i \subseteq \AA^{n_i}$, and a pure polynomial functor $Q_i$
such that $Z_i \subseteq A_i \times Q_i$.
Define $Y$ to be a point, and $B_i \coloneqq Y \times A_i$.
Then the variety $Z_1^\NN \times \cdots \times Z_k^\NN$ of Theorem~\ref{thm:Main2} is a subvariety of the product-type $(\FIop)^k\times \Vec$-variety \[ [Y; B_1\times Q_1, \dots, B_k \times Q_k],\]
with $k=1$ being the special case addressed in Theorem~\ref{thm:main}.
\end{re}

\begin{re} In \cite{draisma-eggermont-farooq:components}, for $\FIop$-varieties (no dependence on $\Vec$), the notion of product type is more restrictive. Essentially,
there the last three authors considered a single finite-dimensional affine variety $Z$ with a
morphism to a finite-dimensional, irreducible, affine variety $Y$, with
the additional requirement that $K[Z]$ is a free $K[Y]$-module. This
then ensures that each irreducible component of $Z^S$ maps dominantly to
$Y$. In \cite{draisma-eggermont-farooq:components} this is used to count the orbits of $\Sym(S)$
on these irreducible components.
\end{re}

The following example describes the partition morphisms between product-type varieties.
It is particularly relevant as this is the shape of the partition morphisms we will be dealing with in our proof of the Parameterisation Theorem~\ref{thm:parameterisation}. 

\begin{ex} \label{ex:PMProductType}
Let $Z'\coloneqq [Y';B_1' \times Q_1',\ldots,B_l' \times Q_l']$ and $Z\coloneqq [Y;B_1
\times Q_1,\ldots,B_k \times Q_k]$ be an $(\FIop)^l \times \Vec$-variety
and an $(\FIop)^k \times \Vec$-variety of product type over
$Y'$ and $Y$, respectively.
We want to construct a partition morphism $(\pi, \phi): Z' \to Z$.
Consider the following data:
\begin{itemize}
    \item let $\pi:[l] \to [k]$ be any map;
    \item let $\alpha:Y' \to Y$ be a morphism of $\Vec$-varieties;
    \item and for each $j \in [l]$ let $\beta_j:B_j' \times Q_j' \to B_{\pi(j)} \times Q_{\pi(j)}$ be a morphism of $\Vec$-varieties such that the following diagram commutes:
    \begin{equation}\label{eq:Rectangle}
    \xymatrix{
    B_j' \times Q_j' \ar[d] \ar[r]^{\beta_j} & B_{\pi(j)} \times Q_{\pi(j)} \ar[d]\\
    Y' \ar[r]_{\alpha} & Y.}
    \end{equation}
\end{itemize}
For each $(T_1,\ldots,T_l) \in \FI^l$ we define the morphism of $\Vec$-varieties
\begin{displaymath}
\phi(T_1,\ldots,T_l)\colon Z'(T_1,\ldots,T_l) \to Z\left (\;\bigsqcup_{j \in
\pi^{-1}(1)}T_j,\ldots,\bigsqcup_{j \in \pi^{-1}(k)} T_j\right)
\end{displaymath}
as follows.
Let $S_i \coloneqq \bigsqcup_{j \in \pi^{-1}(i)}T_j$, then for any $V \in \Vec$ the element
\[
((b'_{j,t},q'_{j,t})_{t \in T_j})_{j \in [l]} \in (B_1' \times Q_1')^{T_1}_{Y'}(V)\times_{Y'} \cdots \times_{Y'}(B_l' \times Q_l')^{T_l}_{Y'}(V)
\]
is mapped to the element
\[
(((\beta_{j}(V)(b'_{j,t},q'_{j,t}))_{t \in
T_j})_{j \in \pi^{-1}(i)})_{i \in [k]} \in (B_1 \times Q_1)^{S_1}_Y (V) \times_Y \cdots \times_Y (B_k \times Q_k)^{S_k}_Y (V).
\]
By construction, the pair $(\pi,\phi)$ is a partition morphism $Z' \to Z$.
Conversely, every partition morphism $Z' \to Z$ is of this form.
Indeed, from a general partition morphism $Z' \to Z$,
$\alpha$ is recovered by taking all $T_j$ empty and $\beta_j$ is recovered
by taking $T_j$ a singleton and all $T_{j'}$ with $j' \neq j$ empty.
That \eqref{eq:Rectangle} commutes then follows
by applying the commuting diagram from the definition of a partition
morphism to the morphism $(\emptyset,\ldots,\emptyset,\ldots,\emptyset)
\to (\emptyset,\ldots,\{*\},\ldots,\emptyset)$ in $\FI^l$. 
\end{ex}

\subsection{The leading monomial ideal}\label{ssec:leadingmonideal}

The following definition gives a size measure for a closed subvariety $B \subseteq Y
\times \AA^n$.

\begin{de}
Let $Y$ be a $\Vec$-variety, $n\in\ZZ_{\geq0}$ and $B$
a closed $\Vec$-subvariety of $Y \times \AA^n$. For $V \in \Vec$ consider
the ideal $\II(B(V))$ of $K[Y(V)][x_1,\ldots,x_n]$ defining $B(V)$.
We fix the lexicographic order on monomials in $x_1,\ldots,x_n$, and
denote by $\LM(B)$ the set of those monomials that appear as leading
monomials of {\em monic} polynomials in $\II(B(V))$, i.e., those with
leading coefficient $1 \in K[Y(V)]$.
\end{de}

Indeed, $LM(B)$ is well-defined:

\begin{lm} \label{lm:LM}
The set $\LM(B)$ does not depend on the choice of $V$.
\end{lm}

\begin{proof}
Let $V \in \Vec$ and consider the linear maps $\iota:0 \to V$ and $\pi:V
\to 0$. If $f \in \II(B(V))$ is monic with leading monomial $x^u$, then
applying $Y(\iota)^\#$ to all coefficients of $f$ yields a polynomial in
$\II(B(0))$ which is monic with leading monomial $x^u$.
This shows that the leading monomials of monic polynomials
in $\II(B(V))$ remain leading monomials of monic elements in $\II(B(0))$.
One obtains the converse inclusion by applying $Y(\pi)^\#$.
\end{proof}

The following lemma monitors the size of $\LM$ of the constant parts after a base change in product-type varieties.
It is used in Proposition~\ref{prop:shift}.

\begin{lm} \label{lm:BaseChangeLM}
Let $Y'\to Y$ be a morphism of $\Vec$-varieties, let $B$ be a closed $\Vec$-subvariety
of $Y \times \AA^n$, and define $B'\coloneqq Y' \times_Y B \subseteq
Y' \times \AA^n$. Then $\LM(B') \supseteq \LM(B)$.
\end{lm}

\begin{proof}
Pulling back a monic equation for $B(V)$ along $Y'(V) \times \AA^n \to
Y(V) \times \AA^n$ yields a monic equation for $B'(V)$ with the same
leading monomial.
\end{proof}

\subsection{Shifting over tuples of finite sets}\label{ssec:shifting}
Shifting over a finite set is a standard technique in the theory of $\FI$-modules \cite{church-ellenberg-farb:FImodules}, and was also used by the last three authors in \cite{draisma-eggermont-farooq:components} to turn certain $\FIop$-varieties into products.
The third author used the operation of shifting over a vector space in \cite{draisma} to prove what became ``The Embedding Theorem'' for $\GL$-varieties in \cite[Theorems~4.1,4.2]{bik-draisma-eggermont-snowden}. 
Here we describe this operation in the context of $(\FIop)^k\times\Vec$-varieties.

\begin{de}
Let $X$ be an $(\FIop)^k \times \Vec$-variety and let
$S=(S_1,\ldots,S_k) \in \FI^k$. Then the {\em shift
$\Sh_{S} X$ of $X$ over $S$} is the $(\FIop)^k \times \Vec$-variety defined by
\[ (\Sh_S X)(T_1,\ldots,T_k)\coloneqq X(S_1 \sqcup T_1,\ldots,S_k
\sqcup T_k) \]
and, for injections $\iota_i:T_i \to T'_i$,
\[ (\Sh_S X)(\iota_1,\ldots,\iota_k)\coloneqq X(\id_{S_1} \sqcup
\iota_1,\ldots,\id_{S_k} \sqcup \iota_k). \qedhere \]
\end{de}

\begin{re}
Consider an tuple $S = (S_1,\dots, S_k)$ in $(\FIop)^k$ and define the covariant functor $\Sh_S: (\FIop)^k\times\Vec\to (\FIop)^k\times\Vec$ by assigning to each tuple $(T_1, \dots, T_k)$ the tuple $(S_1 \sqcup T_1, \dots, S_k \sqcup T_k)$ and to each morphism $ \iota: (\iota_1, \dots, \iota_k): (T_1, \dots, T_k) \to (T_1',\dots, T_k')$ the morphism $\iota \sqcup \id_S$. 
In particular $\Sh_S X$ is the composition $X \circ \Sh_S$.
\end{re}
\begin{re}
Let $V \in \Vec$.
While $\Sh_S X(T_1, \dots, T_k)(V)$ and $X(S_1 \sqcup T_1, \dots, S_k \sqcup T_k)(V)$ coincide as sets, the action induced by functoriality of the $k$ copies of the symmetric group on them is different.
Indeed, the groups $\Sym(T_1)\times \cdots \times \Sym(T_k)$ and $\Sym(S_1 \sqcup T_1) \times \cdots \times \Sym(S_k \sqcup T_k)$ act, respectively, on the former and on the latter.
\end{re}

The following proposition describes the shift operation on product-type varieties.

\begin{prop}\label{prop:shift}
The shift $\Sh_S Z$ over $S=(S_1,\ldots,S_k)$ of an $(\FIop)^k \times \Vec$-variety $Z \coloneqq [Y;B_1 \times Q_1,\ldots,B_k \times Q_k]$ of product
type is itself isomorphic to
a variety of product type: \[\Sh_S Z \cong
[Y'; B_1' \times Q_1,\ldots,B_k' \times Q_k]\] with
\begin{align*}
    Y'&\coloneqq (B_1 \times Q_1)^{S_1}_Y \times_Y \ldots \times_Y  (B_k \times Q_k)^{S_k}_Y, \text{ and}\\
    B_i'&\coloneqq Y' \times_Y B_i.
\end{align*}
Furthermore, each $B'_i$ is
naturally a $\Vec$-subvariety of $Y' \times \AA^{n_i}$, and we have $\LM(B_i')\supseteq \LM(B_i)$.
\end{prop}
\begin{proof}
Straightforward. The last statement follows from
Lemma~\ref{lm:BaseChangeLM}.
\end{proof}




\subsection{Well-founded orders}\label{ss:complexity}
In this paper a {\em pre-order} $\preceq$ on a class is a reflexive and
transitive relation. We also write $B
\succeq A$ for $A \preceq B$. Furthermore, write $A \prec B$ or $B \succ
A$ to mean that $A \preceq B$ but not $B \preceq A$. The
pre-order is well-founded if it admits no infinite
strictly decreasing chains $A_1 \succ A_2 \succ \ldots$.

In this section we first recall a well-founded pre-order
on polynomial functors. Building on it, we define well-founded pre-orders
\begin{itemize}
    \item on varieties appearing in the definition of $(\FIop)^k \times \Vec$-varieties of product type,
    \item on product-type varieties, and
    \item on closed subvarieties of a fixed product-type variety.
\end{itemize}

\subsubsection{Order on polynomial functors}\label{sssec:orderpolfun}
\begin{de}\label{de:orderpolfun}
For polynomial functors $P,Q$, we write $P \preceq Q$ if $P
\cong Q$ or else, for the largest $e$ with $P_e \not \cong
Q_e$, $P_e$ is a quotient of $Q_e$.
\end{de}

This is a well-founded partial
order on polynomial functors, see \cite[Lemma 12]{draisma}.

\subsubsection{Order on $\Vec$-varieties of type $B\times Q$}\label{sssec:orderVecembedded}
Consider $\Vec$-varieties $Y, Y'$, integers $n, n'$, pure polynomial functors $Q, Q'$, and $\Vec$-subvarieties $B \subset Y \times \AA^n$, $B' \subset Y' \times \AA^{n'}$.
We say that $B' \times Q' \preceq B \times Q$ if:
\begin{enumerate}
    \item $Q' \prec Q$ in the order of Definition~\ref{de:orderpolfun}; or
    \item $Q' \cong Q$, $n' = n$ and $\LM(B') \supseteq \LM(B)$.
\end{enumerate}
This is a pre-order on $\Vec$-varieties of this type.

\begin{re}
We remark that $\preceq$ is defined on $\Vec$-varieties {\em with
a specified product decomposition} $B \times Q$ where $B$ is a
$\Vec$-variety {\em with a specified closed embedding into a specified
product} $Y \times \AA^n$ of a $\Vec$-variety $Y$ and some $n$. It is
not a pre-order on $\Vec$-varieties without further data.
\end{re}

\begin{lm} \label{lm:Order}
The pre-order on $\Vec$-varieties defined as above is well-founded.
\end{lm}

\begin{proof}
Suppose we had an infinite strictly decreasing chain
\[ B_1 \times Q_1 \succ B_2 \times Q_2 \succ \ldots \]
with $B_i \subseteq
Y_i \times \AA^{n_i}$. Then we have $Q_1 \succeq Q_2 \succeq
\ldots$. By the well-foundedness of $\succeq$ on polynomial
functors, there exists a $j \geq 1$ such that both $Q_i$ and $n_i$ are constant for $i\geq j$. But then $\LM(B_i) \subsetneq \LM(B_{i+1})\subsetneq \ldots$, which contradicts Dickson's lemma.
\end{proof}

\subsubsection{Order on product-type varieties}
\label{sssec:ProductTypeOrder}
Consider an $(\FIop)^k \times \Vec$-variety $Z \coloneqq [Y; B_1 \times Q_1,
\dots, B_k \times Q_k]$, and an $(\FIop)^l \times \Vec$-variety $Z'\coloneqq
[Y'; B_1' \times Q_1', \dots, B_l' \times Q_l']$. We say that $Z'
\preceq Z$ if there exists a map $\pi:[l] \to [k]$ with the
following properties:
\begin{enumerate}
\item $B_j' \times Q_j' \preceq B_{\pi(j)} \times Q_{\pi(j)}$ holds for all $j \in [l]$, and
\item for all $j$ whose
$\pi$-fibre $\pi^{-1}(\pi(j))$ has cardinality at least $2$ we have $B_j' \times Q_j' \prec B_{\pi(j)} \times Q_{\pi(j)}$.
\item If $\pi$ is a bijection, then either at least one of
the inequalities in (1) is strict, or else $Y'$ is a closed
$\Vec$-subvariety of $Y$.
\end{enumerate}

\begin{lm} \label{lm:Smallerlk}
Suppose $Z' \preceq Z$ is witnessed by $\pi:[l] \to [k]$ and suppose that at least one of the following
holds:
\begin{itemize}
\item $l \neq k$, or
\item at least one of the inequalities in (1) is strict.
\end{itemize}
Then we have $Z' \prec Z$.
\end{lm}

\begin{proof}
Assume, on the contrary, that 
$\sigma:[k] \to [l]$ witnesses $Z \preceq Z'$. Construct a
directed graph $\Gamma$ with vertex set $[l] \sqcup [k]$ and an arrow
from each $j \in [l]$ to $\pi(j)$ and an arrow from each $i \in [k]$ to
$\sigma(i)$. Like any digraph in which each vertex has out-degree $1$,
$\Gamma$ is a union of disjoint directed cycles (here of even length)
plus a number of trees rooted at vertices in those cycles and directed
towards those roots. Moreover, those cycles have the same number of
vertices in $[l]$ as in $[k]$.

The assumptions imply that at least one of the vertices of $\Gamma$ does not lie
on a directed cycle. Without loss of generality, there exists an $i \in
[k]$ not in any cycle such that $j\coloneqq \sigma(i)$ lies on a cycle. Let $n$
be half the length of that cycle, so that $(\sigma \pi)^n (j)=j$. Then
we have
\[ B_j' \times Q_j' \preceq B_{\pi(j)} \times Q_{\pi(j)}
\preceq
\ldots \preceq B_{\pi (\sigma \pi)^{n-1} (j)} \times Q_{\pi
(\sigma \pi)^{n-1}(j)} \prec B_{(\sigma \pi)^n (j)}' \times
Q_{(\sigma \pi)^n (j)}' = B_j' \times Q_j'\]
where the strict inequality holds because $\sigma^{-1}(j)$ has at
least two elements: $i$ and $\pi(\sigma \pi)^{n-1}(j)$. By
transitivity of the pre-order from Section~\ref{sssec:orderVecembedded}, we
find $B_j' \times Q_j' \prec B_j' \times Q_j'$, which however
contradicts the reflexivity of that pre-order. 
\end{proof}

\begin{lm}\label{lm:wellfoundedproduct}
The relation $\preceq$ is a well-founded pre-order on
varieties in $\bC$ of product type.
\end{lm}

\begin{proof}
For reflexivity we may take $\pi$ equal to the identity.
For transitivity, if $\pi:[l] \to [k]$ witnesses $Z' \preceq
Z$ and $\sigma:[k] \to [m]$ witnesses $Z \preceq Z''$, then
$\tau\coloneqq \sigma \circ \pi$ witnesses $Z' \preceq Z''$---here we note
that if $|\tau^{-1}(\tau(j))|>1$ for some $j \in [l]$, then either
$|\pi^{-1}(\pi(j))|>1$ or else $|\sigma^{-1}(\sigma(\pi(j)))|>1$; in
both cases we find that $B_j' \times Q_j' \prec B_{\tau(j)}'' \times
Q_{\tau(j)}''$.

For well-foundedness, suppose that we had a sequence $Z_1 \succ Z_2
\succ Z_3 \succ \ldots$, where
\[ Z_i=[Y_i;B_{i,1} \times Q_{i,1},\ldots,B_{i,k_i} \times
Q_{i,k_i}], \]
and where $\pi_i:[k_{i+1}] \to [k_i]$ is a witness to $Z_i \succ
Z_{i+1}$. We note that $k_i > 0$ for all $i$. Otherwise $0=k_i=k_{i+1}=\ldots$ and
then $Z_i=Y_i \succ Z_{i+1}=Y_{i+1} \succ \ldots $ implies
that $Y_i \supsetneq Y_{i+1} \supsetneq \ldots$, which
contradicts the Noetherianity of the $\Vec$-variety $Y_i$, see Theorem~\ref{thm:Noetherian}.

From the chain, we construct an infinite rooted forest with vertex set $[k_1]
\sqcup [k_2] \sqcup \ldots$ as follows: $[k_1]$ is the set of roots,
and we attach each $j \in [k_{i+1}]$ via an edge with $\pi_i(j)$; the
latter is called the {\em parent} of the former. We further label each
vertex $j \in [k_i]$ with the product $B_{i,j} \times Q_{i,j}$.

We claim that $\pi_i$ is an injection for all $i \gg 0$, i.e., that there
are only finitely many vertices with more than one child. Indeed, if not,
then by K\"onig's lemma the forest would have an infinite path starting
at a root in $[k_1]$ and passing through infinitely many vertices with
at least two children. By construction, the labels $B \times Q$ decrease
weakly along such a path and strictly whenever going from a
vertex to one of its more than one children, a contradiction
to Lemma~\ref{lm:Order}.

For even larger $i$, the $k_i$ are constant, say equal to $k$, and hence
the $\pi_i$ are bijections.  After reordering, we may assume that the $\pi_i$ all equal the identity on $[k]$.  Moreover, for all such $i$ we still have $B_{i,j}
\times Q_{i,j} \succeq B_{i+1,j} \times Q_{i+1,j} \succeq \ldots$ for
all $j \in [k]$, and all these chains stabilise.  When they do, we have
$Y_i \supsetneq Y_{i+1} \supsetneq \ldots$, which is a strictly decreasing chain
of $\Vec$-varieties---but this again contradicts the Noetherianity
of $\Vec$-varieties.
\end{proof}

\subsubsection{Order on closed subvarieties of product-type varieties in $\bC$}\label{sssec:orderclosedproducttype}
Consider the $(\FIop)^k \times \Vec$-variety $Z = [Y; B_1 \times Q_1,
\dots, B_k \times Q_k]$ and let $X$ be a closed $(\FIop)^k \times
\Vec$-subvariety of $Z$; $X$ is not required to be of product type.
We define
\[\delta_X \coloneqq \min_{(S_1,\ldots,S_k)\in\FI^k}\left\{\sum_{i = 1}^k |S_i| :\;
X(S_1, \dots, S_k) \neq Z(S_1,\ldots,S_k)\right\}\]
Let $X$ and $X'$ be closed $(\FIop)^k\times \Vec$-subvarieties of $Z$,
then we say $X' \preceq X$ if $\delta_{X'} \leq \delta_X$.
This is a well-founded pre-order on the $(\FIop)^k\times\Vec$-subvarieties of $Z$.

\begin{re}\label{re:minimaldefiningFItuple}
If $f$ is a nonzero equation for $X(S_1, \dots, S_k)(V)$ with
$\sum_i |S_i|=\delta_X$, then $f$ may still ``come from smaller sets''.
More specifically, there might exist a $k$-tuple $(S_1',
\dots, S_k')$ with $|S_i'|\leq |S_i|$ for all $i \in [k]$
and with strict inequality for at least one $i$, an
$\FI^k$-morphism $\iota \coloneqq (\iota_1, \dots, \iota_k):
(S_1', \dots, S_k') \to (S_1, \dots, S_k)$, and an element
$f'\in K[Z(S_1', \dots, S_k')(V)]$ such that $Z(\iota)(V)^\#(\;f')
= f$.
This is related to Remark~\ref{re:Richer}.
The following example demonstrates this phenomenon.
\end{re}
\begin{ex}\label{ex:fewer}
Consider the $\FIop \times \Vec$-variety $Z \coloneqq [\Spec(K) ; \AA^1]$.
The coordinate ring $K[Z(S)]$ is isomorphic to the polynomial ring over $K$ in $|S|$ variables.
Let $n\in\ZZ_{>0}$ and define the proper closed variety $X$ of $Z$ by
\[
X(S) \coloneqq \begin{cases}
Z(S) & \text{ for } |S| < n;\\
\emptyset & \text{ otherwise}.
\end{cases}
\]
Then $\delta_X$ is equal to $n$ and computed by the element
$1 \in K[Z([n])]$, which is the image of $1 \in
K[Z(\emptyset)]$ under the natural map $K[Z(\emptyset)] \to K[Z([n])]$.
\end{ex}

\section{Covering $(\FIop)^k \times \Vec$-varieties by
smaller ones}  \label{sec:Core}

\subsection{The Parameterisation Theorem}
The goal of this section is to prove the following core
result, which says that any proper closed subvariety of an
$(\FIop)^k \times \Vec$-variety of product type is covered
by finitely many smaller such varieties.

\begin{thm}[Parameterisation Theorem]\label{thm:parameterisation}
Consider an $(\FIop)^k \times \Vec$-variety $Z$ of product type and let $X \subsetneq Z$ be a proper closed $(\FIop)^k \times \Vec$-subvariety.
Then there exist a finite number of quadruples consisting of:
\begin{itemize}
\item an $l \in \ZZ_{\geq 0}$;
\item an $(\FIop)^l \times \Vec$-variety $Z'$ of product
type with $Z' \prec Z$;
\item a $k$-tuple $S = (S_{1}, \dots, S_{k}) \in \FI^{k}$; and
\item a partition morphism $(\pi, \phi) : Z' \to \Sh_{S} Z$;
\end{itemize}
such that for any $T_1,\ldots,T_k \in \FI^k$, any $V \in
\Vec$, and any $p \in X(T_1,
\ldots, T_k)(V)$ there exist: one of these finitely many quadruples; finite
sets $U_1,\ldots,U_k$; and bijections $\sigma_i:T_i \to S_i
\sqcup U_i$; such that $p$ lies in the image under
$Z(\sigma_1,\ldots,\sigma_k)(V)$ of the image of $(\pi,\phi)$
in $\Sh_S(Z)(U_1,\ldots,U_k)(V)=Z(S_1 \sqcup U_1,\ldots,S_k
\sqcup U_k)(V)$.
\end{thm}

\begin{re}\label{re:unwrap}
Recall Definition~\ref{de:Image} of the image of a partition
morphism. Explicitly, the conclusion above means that there
exist finite sets $U_1',\ldots,U_l'$ and, for each $i \in
[k]$, a bijection $\iota_i:U_i \to \bigsqcup_{j \in \pi^{-1}(i)}
U_j'$, and a point $q \in Z'(U_1',\ldots,U_l')(V)$ such that
\[
(Z(\sigma_1,\ldots,\sigma_k)(V) \circ
(\Sh_SZ)(\iota_1,\ldots,\iota_l)(V) \circ
\phi(U_1',\ldots,U_l')(V))(q) = p.
\]
Informally, we will say that all points in $X$ are {\em hit} by
finitely many partition morphisms from varieties $Z'$ in
$\bC$ of product type with $Z' \prec Z$.
\end{re}

\subsection{A key proposition}

The proof of Theorem~\ref{thm:parameterisation} uses a key proposition that we establish first.
The reader may prefer to read only the statement of this proposition and postpone its proof until after reading the proof of Theorem~\ref{thm:parameterisation} in Section~\ref{ssec:ProofParm}.

\begin{prop} \label{prop:Smaller}
Let $Y$ be a $\Vec$-variety; $n\in\ZZ_{\geq0}$; $B$ a closed
$\Vec$-subvariety of $Y \times \AA^n$; $Q$ a pure polynomial functor;
and $X$ a proper closed $\Vec$-subvariety of $B \times Q
\subseteq Y \times \AA^n \times Q$.
Then there exist:
\begin{itemize}
\item a proper closed $\Vec$-subvariety $Y_0$ of $Y$;
\item a $\Vec$-variety $Y'$ together with a morphism $\alpha:Y' \to Y$;
\item $k \in \ZZ_{ \geq 0}$;
\end{itemize}
and for each $l = 0 ,\dots, k$: 
\begin{itemize}
    \item an
integer $n_l \in \ZZ_{\geq 0}$;
\item a 
closed $\Vec$-subvariety $B_l \subseteq Y' \times \AA^{n_l}$;
\item a pure polynomial functor $Q_l$; 
\item and a morphism $\beta_l:B_l \times Q_l \to B \times Q,$
\end{itemize}
such that the following properties hold:
\begin{enumerate}
\item\label{thesisone} For each $l = 0, \dots, k$, $B_l \times Q_l \prec B
\times Q$ in the preorder from Section~\ref{sssec:orderVecembedded}, and
the following diagram commutes:
\[
\xymatrix{
B_l \times Q_l \ar[r]^{\beta_l} \ar[d] & B \times Q \ar[d]\\
Y' \ar[r]_{\alpha} & Y.}
\]
\item\label{thesistwo} Let $m \in \ZZ_{\geq 0}$,  $V \in \Vec$, and points $p_1,\ldots,p_m
\in X(V) \subseteq Y(V) \times \AA^n \times Q(V)$ whose images in $Y(V)$
are all equal to the same point $y \in Y(V) \setminus Y_0(V)$.
Then there exist indices
$l_j \in \{0, \dots, k\}$ for $j \in [m]$ and points $p'_j \in B_{l_j}(V) \times
Q_{l_j}(V)$ whose images in $Y'(V)$ are all equal to the same point
$y'$ and such that $\beta_{l_j}(V)(p'_j)=p_j$ for all $j \in [m]$.
\end{enumerate}
\end{prop}

\begin{re}
The condition $\beta_{l_j}(V)(p'_j)=p_j$, together with the
commuting diagram in (\ref{thesisone}), implies $\alpha(y')=y$.
\end{re}

\begin{re}
    The labelling by $l \in \{0,\ldots,k\}$ rather than by $l \in [k]$ is chosen because in the proof of Proposition~\ref{prop:Smaller} the data for $l=0$ are chosen in a slightly different manner than those for $l>0$. However, in the statement of that proposition, all $l$ play equivalent roles.
\end{re}

To apply Proposition~\ref{prop:Smaller} in the proof of
Theorem~\ref{thm:parameterisation} we will do a shift over an
appropriate $k$-tuple of finite sets. After this shift,
we deal with the points of $X$ lying over $Y_0$ by induction, while we
cover those in the complement by a partition morphism constructed with
the morphisms $\alpha$ and $\beta_j$'s, and whose domain is a product-type
variety strictly smaller than $Z$.
Before proving Proposition~\ref{prop:Smaller} in
Section~\ref{ssec:Smaller}, we demonstrate its statement in two special cases.

\begin{ex}
Consider the case where $Y=\Spec K$ and $n=0$; then $B \subseteq
Y \times \AA^n$ is also isomorphic to $\Spec K$. Let $Q$ be an arbitrary polynomial functor. In this case,
$X$ is a proper closed $\Vec$-subvariety of $Q$ and by \cite{bik-draisma-eggermont-snowden}
there exist $k\in\ZZ_{\geq 0}$, (finite-dimensional) varieties $B_0,\ldots,B_k$,
pure polynomial functors $Q_0,\ldots,Q_k \prec Q$ and morphisms $\beta_l:
B_l \times Q_l \to Q$ such that $X$ is the union of the images of the
$\beta_l$.
This is an instance of Proposition~\ref{prop:Smaller} with $Y_0=\emptyset$, $Y'=Y$, and $\alpha=\id_Y$.
Note that then $B_l \times Q_l \prec Q$ since $Q_l \prec Q$, so the specific choice of embedding $B_l \subseteq \AA^{n_l}$ is not relevant.
\end{ex}

\begin{ex}
Consider the case where $Y$ is constant, that is, just given by a (finite-dimensional) variety, and
$Q=0$. Since $X$ is a proper closed subvariety of $B
\subseteq Y \times \AA^n$, there exist a $V\in \Vec$ and a nonzero function
$f \in K[B(V)]$ that vanishes identically on $X(V)$.

Then $f$ is represented by a polynomial in
$K[Y(V)][x_1,\ldots,x_n]$, also denoted by $f$. We may reduce
$f$ modulo $\II(B(V))$
in such a manner that its leading term $c \cdot x^u$ has the property
that $c \in K[Y(V)]$ is nonzero and $x^u \not \in \LM(B)$. Then we take
for $Y_0$ the closed subvariety of $Y$ defined by the vanishing of $c$
and for $Y'$ the complement $Y \setminus Y_0$, with $\alpha:Y' \to Y$
being the inclusion. Furthermore, we take $k=0$,
and $B_0$ to be the intersection of $B$ with
$Y' \times \AA^n$ and with the vanishing locus of $f$ in $Y \times \AA^n$.
Then $\LM(B_0) \supseteq \LM(B)$ and since $c$ is
invertible on $Y'$ and $f$ vanishes on $B_0$, $x^u \in \LM(B_0) \setminus \LM(B)$. To verify (2) of Proposition~\ref{prop:Smaller},
we observe that the $p_j$ all map to the same point in $Y'=Y \setminus
Y_0$, i.e., $p_j$ lies in the set $B_0 \subseteq B$, and we can
just take $p'_j\coloneqq p_j$ for all $j$.
\end{ex}

\subsection{Iterated partial derivatives}
The main technical result for proving Proposition~\ref{prop:Smaller} is Lemma~\ref{lm:Dra19} below. This is essentially an iteration of the argument used to establish the Embedding Theorem in \cite{bik-draisma-eggermont-snowden}, which involves directional derivatives of a function defining a $\Vec$-variety along a direction lying in an irreducible subobject of the 
top-degree part of the ambient polynomial functor.

\begin{lm}
\label{lm:Dra19}
Let $B$ be a $\Vec$-variety and $Q$ a pure polynomial
functor. Decompose 
\[
Q = R_1 \oplus \cdots \oplus R_t,
\]
where the $R_i$ are irreducible objects in the abelian
category of polynomial functors, arranged in weakly
increasing degrees. 
Denote with $R_{\leq s}$ the functor $\bigoplus_{i = 1}^s
R_i$, so that $R_{\leq 0}=0$. 
Let $X$ be a proper closed $\Vec$-subvariety of $B  \times
Q$. Then there exist
\begin{itemize}
    \item a $k \in \ZZ_{\geq 0}$;
    \item $U_0, \dots, U_k \in \Vec$ with partial sums
    $U_{\leq s} \coloneqq \bigoplus_{i = 0}^s U_i$ for $s
    \geq 0$;
    \item indices $0 = s_0 < s_1 \leq \cdots \leq s_{k} \leq t$;
    \item for each $l \in \{0,\ldots,k\}$ a nonzero function 
    $h_l \in K[B(U_{\leq l})\times R_{\leq s_l}(U_{\leq
    l})]$ (so that $h_0 \in K[B(U_0)]$); and 
    \item for each $l \in \{1,\ldots,k\}$, a nonzero
    coordinate $x_l \in R_{s_l}(U_{l})^*$ and a function $r_l$ in $K[B(U_{\leq l})\times (R_{\leq s_l}(U_{\leq l})/R_{s_l}(U_l))]$ such that 
\[
h_l=x_l \cdot h_{l-1} + r_l;
\]
\end{itemize}
and such that, moreover, the function $h_k$ vanishes on $X(U_{\leq k})$.
\end{lm}

\begin{re}
It is here that we use the fact that $K$ has characteristic zero, in at least two different ways: the fact that an arbitrary polynomial functor is a direct sum of irreducible ones, and the fact that, by acting with the Lie algebra of $\GL_n$, we can go from an equation to an equation of weight $(1,\ldots,1)$. We think that our main theorem may be true in positive characteristic as well, but the proof would be more technical and involve techniques from \cite{bik-draisma-snowden-pos-char}, where the theory of $\GL$-varieties in positive characteristic is developed. 
\end{re}

\begin{proof}
    Let $U$ be a finite-dimensional vector space for which there exists
a nonzero $f \in K[B(U) \times Q(U)]$ that vanishes identically on
$X(U)$. Without loss of generality, $U=K^n$ for some $n$. Since the
vanishing ideal of $X(U)$ is a $\GL(U)$-module, we may assume that
$f$ is a weight vector with respect to the standard maximal torus in
$\GL(U)=\GL_n$. Furthermore, by enlarging $U$ if necessary ($n=\deg(f)$
suffices)  we may assume that the weight of $f$ is $(1,\ldots,1)$
(see \cite[Lemma 3.2]{snowden:stable}; strictly speaking, our $\GL(U)$-action is contragredient to the action there, and writing $(-1,\ldots,-1)$ would be more consistent).

Choose $s_k$ as the maximal index in $[t]$ such that $f$ involves
coordinates in $R_{s_k}(U)^*$; if no such index exists, then $k$ is set to
zero, and we may take $U_0=U$ and $h_0=f \in K[B(U_0)]$ and we are done.

After acting with the symmetric group $\Sym([n])$ if necessary, we may
assume that $f$ contains at least one coordinate in $R_{s_k}(U)^*$
of weight $(0,\ldots,0,1,\ldots,1) \eqqcolon (0^{n'},1^{n_k})$, where there
are $n'$ zeroes and $n_k$ ones, with $n'+n_k=n$. We set $U'\coloneqq K^{n'}$
and $U_k\coloneqq K^{n_k}$, so that $U=U' \oplus U_k$. Since $f$ has weight
$(1,\ldots,1)$, we can decompose
\[ f=\left(\sum_{i=1}^N f_i \cdot y_i\right) + r \]
where $N \geq 1$, the $f_i$ have weight $(1^{n'},0^{n_k})$; the $y_i$
are elements in $R_{s_k}(U)^*$ of weight $(0^{n'},1^{n_k})$ and
hence lie in $R_{s_k}(U_k)^*$; and $r$ does not contain elements in
$R_{s_k}(U_k)^*$. This implies that the $f_i$ are elements of $K[B(U')
\times R_{\leq s_k}(U')]$ and $r$ is an element of $K[B(U' \oplus U_k)
\times (R_{\leq s_k}(U' \oplus U_k)/R_{s_k}(U_k))]$.
Furthermore, we may assume that the $f_i$ are linearly independent over
$K$. 

Now act on $f$ with upper triangular elements of $\gl(U_k)$.  With
respect to this action, the $f_i$ are constants, the $y_i$ are replaced by
higher-weight vectors in $R_{s_k}(U_k)^*$, and $r$ remains an element of
$K[B(U' \oplus U_k) \times (R_{\leq s_k}(U' \oplus U_k)/R_{s_k}(U_k))]$.
We can choose a sequence of such upper triangular elements that takes
$y_1$ to a nonzero highest weight vector $v$ in $R_{s_k}(U_k)^*$, and
the same sequence will take each $y_i$ to a scalar multiple of $v$.
Since the $f_i$ are linearly independent, the term $f_1 \cdot v$ in
the result is not cancelled. Hence after this action, $f$ has been
transformed to the desired shape 
\[ f=h \cdot x_k + r \]
with $h \in K[B(U') \times R_{\leq s_k}(U')]$, $x_k$ a nonzero
highest weight vector in $R_{s_k}(U_k)^*$ and $r$ lies in the ring \[ K[B(U' 
\oplus U_k) \times (R_{\leq s_k}(U' \oplus U_k)/R_{s_k}(U_k))].\]
Now we treat the pair $(U',h)$ in exactly the same manner as we treated
the pair $(U,f)$, dragging $r$ along in the process: pick $s_{k-1}$
maximal such that $h$ contains elements from $R_{s_{k-1}}(U')^*$. By
acting with the symmetric group $\Sym([n'])$ on $f$ we may assume
that $h$ contains an element from $R_{s_{k-1}}(U')^*$ of weight
$(0^{n''},1^{n_{k-1}})$, with $n''+n_{k-1}=n'$. Then set $U''=K^{n''}$
and $U_{k-1}=K^{n_{k-1}}$, so that $U'=U'' \oplus U_{k-1}$. 
By acting on $f$ with upper triangular elements of $\gl(U_{k-1})$
we transform it into the shape
\[ f=(\tilde{h} \cdot x_{k-1} + \tilde{r}) \cdot x_k + r \]
where $x_k$ has not changed, $r$ has changed within the space $K[B(U'
\oplus U_k) \times (R_{\leq s_k}(U' \oplus U_k)/R_{s_k}(U_k))]$, $x_{k-1}$
is a highest weight vector in $R_{s_{k-1}}(U_{k-1})^*$, $\tilde{h}$
lies in $K[B(U'') \times R_{\leq s_{k-1}}(U'')]$, and $\tilde{r}$
lies in the ring \[K[B(U'' \oplus U_{k-1}) \times (R_{\leq s_{k-1}}(U'' \oplus U_{k-1})/R_{s_{k-1}}(U_{k-1}))].\]
Continuing in this fashion, we eventually put $f$ in the form
\begin{equation*}
f=x_k(x_{k-1}(\dots(x_2(x_1 h_0 +
r_1)+r_2)\dots)+r_{k-1})+r_k
\end{equation*}
where $h_0 \in K[B(U_0)]$ and $U_0$ is the space left over from $U$ after splitting off all the $U_i$ with $i>0$. Now set
\[ h_l \coloneqq x_l(x_{l-1}(\cdots(x_2(x_1h_0+r_1)+r_2)\ldots)+r_{l-1})+r_l\]
and we are done.
\end{proof}

\subsection{Proof of Proposition~\ref{prop:Smaller}}
\label{ssec:Smaller}

This section contains the proof of the Proposition~\ref{prop:Smaller}, and, for clarity's sake, we spell it out in a concrete example at the end.

\begin{re}
We recall that, for any $\Vec$-variety $Z$ and any $U \in \Vec$, the
shift $\Sh_U Z$ of $Z$ over $U$ is the $\Vec$-variety defined by $(\Sh_U
Z)(V)=Z(U \oplus V)$. There is a {\em natural morphism} $\Sh_U Z \to Z$
of $\Vec$-varieties: for $V \in \Vec$, this morphism $(\Sh_U Z)(V)=Z(U
\oplus V) \to Z(V)$ is just $Z(\pi_V)$, where $\pi_V$ is the projection
$U \oplus V \to V$.
\end{re}

\begin{lm} \label{lm:LM2}
Let $Y$ be a $\Vec$-variety, $n\in\ZZ_{\geq0}$, and $B$
a closed $\Vec$-subvariety of $Y \times \AA^n$. Then for any $U
\in \Vec$, $\Sh_U B$ is a closed $\Vec$-subvariety of $(\Sh_U
Y) \times \AA^n$, and $\LM(B)=\LM(\Sh_U(B))$.
\end{lm}

\begin{proof} This follows from Lemma~\ref{lm:BaseChangeLM}.
\end{proof}

\begin{re}
    Let $X$ be a $\Vec$-variety, $U\in\Vec$ and $f \in K[X(U)]$. We define $(\Sh_U X)[1/f]$ to be the $\Vec$-variety given by $V \mapsto X(U\oplus V)[1/f]$, where we identify $f$ with its image under the natural map $K[X(U)] \to K[X(U \oplus V)]$.
    Note that the action of the group $\GL$ on the coordinate ring of $\Sh_U X$ is the identity on the element $f$.
    In particular, for every $V \in\Vec$, $(\Sh_U X[1/f])(V) \subseteq \Sh_U X(V)$ is the distinguished open set of points not vanishing on the single $f$.
\end{re}

\begin{proof}[Proof of Proposition~\ref{prop:Smaller}]
Since $X$ is a proper closed subvariety of $B \times Q$, we apply the machinery of Lemma~\ref{lm:Dra19}. 
Decompose $Q$ as $R_1 \oplus \cdots \oplus R_t$, where the $R_s$ are
irreducible polynomial functors and $\deg(R_s) \leq \deg(R_{s+1})$ for all $s=1,\ldots,t-1$.
Write $R_{\leq s}\coloneqq R_1 \oplus \cdots \oplus R_s$ and $R_{>s} \coloneqq R_{s+1} \oplus \cdots \oplus R_t$, so that $R_{\leq 0}=\{0\}$ and $R_{>t} = \{0\}$.

By Lemma~\ref{lm:Dra19}, we can construct a sequence of vector spaces $U_0,U_1,\ldots,U_k$
with partial sums $U_{\leq l}\coloneqq \bigoplus_{i=0}^l
U_i$, indices $0=s_0<s_1 \leq
\cdots \leq s_k \leq t$, nonzero coordinates $x_l \in R_{s_l}(U_l)^*$
for $l \in [k]$, nonzero functions $h_l \in K[B(U_{\leq l}) \times R_{\leq s_l}(U_{\leq
l})]$ for $l=0,\ldots,k$ and functions $r_l \in K[B(U_{\leq l}) \times
(R_{\leq s_l}(U_{\leq l})/R_{s_l}(U_l))]$
for $l \in [k]$ such that
\begin{equation}\tag{A}\label{eq:hl}
h_l=x_l \cdot h_{l-1} + r_l
\end{equation}
for each $l=1,\ldots,k$ and such that $h_k$ that vanishes on
$X(U_{\leq k})$.

Now $h_0 \in K[B(U_0)]$ is represented by a polynomial in
$K[Y(U_0)][x_1,\ldots,x_n]$, and after reducing modulo $\II(B(U_0))$,
we may assume that its leading term equals $c \cdot x^u$ where $c \in
K[Y(U_0)]$ is nonzero and $x^u \not \in \LM(B)$.

Now set $U\coloneqq U_{\leq k}=U_0 \oplus \cdots \oplus U_k$. Then we
construct the relevant data as follows.
\begin{enumerate}
    \item\label{Yzero} Define $Y_0$ as the closed $\Vec$-subvariety of $Y$ defined by
    the vanishing of $c$, so that 
    \[ Y_0(V)\coloneqq \{y \in Y(V) \mid \forall \phi \in \Hom_{\Vec}(V,U_0):
    c(Y(\phi)y)=0\}. \]
    
    \item Set $Y'\coloneqq (\Sh_U Y)[1/c]$ with $\alpha:Y' \to Y$ the restriction
    to $Y'$ of the natural morphism $\Sh_U Y \to Y$. 
    \item\label{Bzero} Let $B_0$ be the closed $\Vec$-subvariety of $(\Sh_U B)[1/c]$
    defined by the vanishing of the single equation $h_0$. 
    Note that
    $B_0$ is a closed $\Vec$-subvariety of $Y' \times \AA^{n_0}$ with $n_0\coloneqq n$.
    Define $Q_0\coloneqq Q$ and $\beta_0:
    B_0 \times Q_0 \to B \times Q$ as the identity on $Q$ and equal to
    the restriction to $B_0$ of the natural morphism $\Sh_U B \to B$
    on $B_0$.  Note that $\LM(B_0) \supseteq \LM(B)$ by virtue of
    Lemma~\ref{lm:LM2}, and since $h_0 \in \II(B_0(U_0))$ has
    leading term $c \cdot x^u$ and $c$ is invertible on $Y'$,
    we have $x^u \in \LM(B_0) \setminus \LM(B)$. Thus $B_0 \times
    Q_0 \prec B \times Q$.
    
    \item For $l\in[k]$, set
    \[ Q_l\coloneqq ((\Sh_U R_{\leq s_l})/(R_{\leq s_l}(U) \oplus R_{s_l})) \oplus R_{>s_l}. \]
    Here we recall that, for any pure polynomial functor $R$, the top-degree
    part of $\Sh_U R$ is naturally isomorphic to that of $R$, and its constant
    part is isomorphic to $R(U)$ (see \cite[Lemma 14]{draisma}
    for the first statement; the second is proved in a similar
    fashion). So, since we ordered the irreducible factors
    $R_s$ by ascending degrees, $R_{s_l}$ is naturally a sub-object of the
    top-degree part of $\Sh_U R_{\leq s_l}$; and the constant polynomial
    functor $R_{\leq s_l}(U)$ is the constant part of $\Sh_U R_{\leq s_l}$. Both
    are modded out, and we have $Q_l \prec Q$.
    
    \item For $l\in [k]$, we define $B_l$ as
    \begin{align*}
    B_l&\coloneqq (\Sh_U B)[1/c] \times R_{\leq s_l}(U) \times
    \AA^1 \\
    &\subseteq Y' \times \AA^n \times R_{\leq s_l}(U) \times
    \AA^1 \cong Y' \times \AA^{n_l}.
    \end{align*}
    where $n_l\coloneqq n+\dim(R_{\leq s_l}(U))+1$. Note that the factor $R_{\leq
    s_l}(U)$ is precisely the constant term modded out in the definition
    of $Q_l$; the role of the factor $\AA^1$ will become clear
    below.
    
    \item To construct $\beta_l:B_l \times Q_l \to B \times Q$ we proceed as
    follows. 
    Let $X_l$ be the closed
    $\Vec$-subvariety of $B \times R_{\leq s_l}$ defined by the vanishing of $h_l$.
    Then \eqref{eq:hl} shows that, on the distinguished open
    subvariety $(\Sh_{U_{\leq l-1}} X_l)[1/h_{l-1}]$, the coordinate $x_l$ can be expressed
    as a function on $\Sh_{U_{\leq l-1}} B \times ((\Sh_{U_{\leq l-1}} R_{\leq s_l})
    / R_{s_l})$ evaluated at $U_l$. Since $R_{s_l}$ is
    irreducible, {\em each} coordinate on it can be thus
    expressed; this is a crucial point in the proof of
    \cite[Lemma 25]{draisma}.
    This implies that the projection
    \[ \Sh_{U_{\leq l-1}} B \times \Sh_{U_{\leq l-1}} R_{\leq s_l} \to
    (\Sh_{U_{\leq l-1}} B) \times (\Sh_{U_{\leq l-1}} R_{\leq s_l})/R_{s_l} \]
    restricts to a closed immersion of $(\Sh_{U_{\leq l-1}} X_l)[1/h_{l-1}]$ into the open
    subvariety of the right-hand side where $h_{l-1}$ is nonzero. This
    statement remains true when we replace $U_{\leq l-1}$ everywhere by the
    larger space $U$. After also inverting $c$, we find a closed immersion
    \[ (\Sh_U X_l)[1/h_{l-1}][1/c]
    \to
    (\Sh_U B)[1/c] \times (\Sh_U R_{\leq s_l})/R_{s_l} \times \AA^1,
    \]
    where the map to the last factor is given by $1/h_{l-1}$. By \cite[Proposition 1.3.22]{bik} the inverse morphism from the image of this closed immersion lifts to a morphism of ambient $\Vec$-varieties
    \begin{align*}
    \iota: &B_l \times (\Sh_U R_{\leq s_l})/(R_{\leq s_l}(U) \oplus
    R_{s_l}) \\
    & \cong (\Sh_U B)[1/c] \times (\Sh_U R_{\leq s_l})/R_{s_l} \times \AA^1 \\
    &\to
    \Sh_U (B \times R_{\leq s_l})
    \end{align*}
    that hits all the points in $(\Sh_U X_l)[1/h_{l-1}][1/c]$. Finally, we
    define $\beta_l \coloneqq \beta_l' \times \id_{R_{>s_l}}$ where $\beta_l'$ is
    the composition of $\iota$ and the natural
    morphism $\Sh_U (B \times R_{\leq s_l}) \to B \times R_{\leq s_l}$.
\end{enumerate}
Property (1) in the proposition holds by construction. We
now verify property (2). Thus let $V \in \Vec$, $m \in
\ZZ_{\geq 0}$, and let $p_1,\ldots, p_m \in X(V) \subseteq
Y(V) \times \AA^n \times Q(V)$. Assume that the images
of $p_1,\ldots, p_m$ in $Y(V)$ are all equal to $y$, and that
$y \not \in Y_0(V)$. By definition of $Y_0$, this means that
there exists a $\phi \in \Hom_{\Vec}(V,U)$ such that
$c(Y(\phi)(y)) \neq 0$.

On the other hand, we have $h_k(X(\psi)(p_j))=0$ for all $j$
and all
$\psi:V \to U$, because
$h_k$ vanishes identically on $X$. For $j \in [k]$ define
\[
  l_j \coloneqq \min\{l \mid \forall \psi \in \Hom_{\Vec}(V,U) : h_l(X(\psi)(p_j)) = 0\}.
\]
Put differently, $l_j$ is the smallest index $l$ such that the projection of $p_j$ in $B \times R_{\leq s_l}$ lies in $X_l \subseteq B \times R_{\leq s_l}$. Note that, if $l_j>0$, then there
exists a linear map $\psi:V \to U$ such that $h_{l_j-1}(X(\psi)(p_j))
\neq 0$.

Since $\Hom_\Vec(V,U)$ is irreducible, there exists a linear map $\phi \colon V\to U$ such that first, $c(Y(\phi)(y)) \neq 0$; and second, $h_{l_j-1}(X(\phi)(p_j)) \neq 0$ for all $j$ with $l_j>0$.

We now define the $p_j'$ as follows. First, we decompose $p_j=(p_{j,1},p_{j,2})$
where $p_{j,1} \in B(V) \times R_{\leq s_{l_j}}(V)$ and $p_{j,2} \in
R_{>s_{l_j}}(V)$. Similarly, we decompose the
point $p_j'=(p_{j,1}',p_{j,2}')$ to be constructed.
\begin{enumerate}
\item Set $p_{j,2}'\coloneqq p_{j,2}$ for all $j$.  Recall that we had defined
$s_0\coloneqq 0$, so that this implies that if $l_j=0$, then the component
$p_{j,2}'$ of $p_j'$ in $Q$ equals the component $p_{j,2}$ of $p_j$
in $Q$.

\item If $l_j=0$, then $p_{j,1} \in B(V)$, and $p_{j,1}' \in
B_0(V) \subseteq (\Sh_U B)[1/c](V)$ is defined as  $B(\phi \oplus
\id_V)(p_{j,1})$.  Note that $p_{j,1}'$ does indeed lie in $B_0(V)$;
this follows from the fact $l_j=0$, so that $h_0(B(\psi)(p_{j,1}))=0$
for all $\psi:V \to U_0$, and hence also for all $\psi$ that decompose
as $\psi' \circ (\phi \oplus \id_V)$.

Furthermore, note that $\beta_0(V)(p_j')=p_j$; this follows from the
equality $\pi_V \circ (\phi \oplus \id_V)=\id_V$. Also, the image of
$p_j'$ in $Y'(V)$ equals $Y(\phi \oplus \id_V)(y)\eqqcolon
y'$.

\item If $l\coloneqq l_j>0$, then $p_{j,1} \in B(V) \times R_{\leq s_l}(V)$ with
$s_l \geq 1$, and $p_{j,1}'$ is constructed as follows. First apply $(B
\times R_{\leq s_l})(\phi \oplus \id_V)$ to $p_{j,1}$ and then forget
the component in $R_{s_l}(V)$. The morphism $\beta_l'$ was constructed
in such a manner that $\beta_l'(V)(p_{j,1}')=p_{j,1}$ and
therefore $\beta_l(V)(p_j')=p_j$. Note that also the image of
$p_j'$ in $Y'(V)$ equals $y'$. This concludes the proof.
\qedhere
\end{enumerate}
\end{proof}


\begin{ex}\label{explanatory}
Write $Y$ for the polynomial functor $V \to V \oplus V$ and write $K[x_i,
y_i \mid i \in [n]]$ for the coordinate ring of $Y(K^n)$.  Consider the
$\Vec$-subvariety $B$ of $Y \times \AA^1$ defined by $y_1-t \cdot x_1$,
where $t$ is the coordinate of $\AA^1$. Then $\LM(B)=\emptyset$ and
$B(V)$ is the set of triples $(v, \lambda v, \lambda)$ with $v \in V$
and $\lambda \in K$. Set $Q(V)\coloneqq S^2 V$, and choose
coordinates $z_{ij}, i \leq j$ on $Q(K^n)$ by 
writing an arbitrary element of $Q(K^n)$ as 
\[ \sum_{i = 1}^n z_{ii} e_i^2 + \sum_{ 1\leq i <j\leq n} 2 z_{ij} e_i e_j.\]
Note that $Q$ is an irreducible polynomial functor, so, in the notation of Proposition~\ref{prop:Smaller}, we have $R = R_1 = Q$.
Define the $\Vec$-subvariety
\begin{align*}
&X \subset B \times Q \subset Y \times \AA^1 \times Q \;\text{ by }\\
&X(V)\coloneqq \{(v,w,\lambda,q) \mid (v,w,\lambda) \in B(V) \text{ and }
w^2,q \text{ are linearly dependent}\}.
\end{align*}
An equation for $X(K^2)$ is the determinant
\[
f\coloneqq z_{12} y_1^2 - z_{11} y_1y_2 = t^2 ( z_{12}x_1^2 - z_{11} x_1 x_2
)\in K[B(U_0)\times Q(U_0)]
\]
with $U_0\coloneqq K^2$. Define $U_1\coloneqq \langle e_3, e_4 \rangle \cong K^2$,
so that $U_0 \oplus U_1=K^4$.
Acting on $f$ equation with the (upper triangular) elements $E_{1,3}$ and $E_{2,4}$ of the Lie algebra $\gl(U_0 \oplus U_1)$ gives the equation:
\begin{displaymath} 
h_1 \coloneqq z_{34} (x_1^2 t^2) + (2 z_{14} x_1 x_3 - 2
z_{13} x_1 x_4 - z_{11} x_3 x_4) t^2
\end{displaymath}
that, by construction, vanishes on $X(U_0 \oplus U_1)$.  Note that $z_{34}
\in Q(U_1)^*$, $h_0 \coloneqq x_1^2 t^2 \in K[B(U_0)]$ (and we let $c$
be the leading coefficient: $c \coloneqq x_1^2$), and the rest belongs
to $K[B(U_0 \oplus U_1) \times Q(U_0 \oplus U_1)/Q(U_1)]$. 

By acting with permutations $(3,i)$ and $(4,j)$ with $i<j$ on $h_1$ we find that,
where $h_0$ is nonzero, on $X$ we have
\begin{equation} \label{eq:zij} 
z_{ij} = -\frac{1}{h_0}  \cdot  (2 z_{1j} x_1 x_i - 2
z_{1i} x_1 x_j - z_{11} x_i x_j) t^2. \end{equation}
A similar expression can be found for $z_{ii}$, with the
same denominator $h_0$. 

In this case, $Y_0$ from the proposition is the $\Vec$-subvariety of $Y$
defined by $c=x_1^2$. This consists of all pairs $(0,w) \in V \oplus
V$. The preimage in $X$ consists of all quadruples $(0,0,\lambda,q)$
with $q$ arbitrary.

Set $U \coloneqq U_0 \oplus U_1$, $Y' \coloneqq \Sh_U Y[1/c]$, and let
$B_0$ be the vanishing locus of $h_0$ in $\Sh_U B[1/c] \subset  Y'
\times \AA^1$. Note that we have $t^2 \in \LM(B_0)$---indeed, $t$ even
vanishes identically on $B_0$.
With $Q_0 \coloneqq Q$ we find $B_0 \times Q_0 \prec B\times
Q$, and we define the map:
\[
\beta_0 : B_0 \times Q_0 \to B \times Q
\]
as $B(\pi_V)|_{B_0} \times \id_{Q(V)}$ for every $V \in \Vec$. This covers
all the points in $X(V)$ of the form $(v,0,0,q)$ with $v,q$ arbitrary.

Finally, consider the map:
\[
\Sh_U(B\times Q)[1/h_0][1/c] \to \Sh_U(B\times Q)/Q \times \AA^1
\cong (\Sh_U B \times Q(U) \times \AA^1) \times (\Sh_U
Q/(Q(U) \oplus Q)) \eqqcolon B_1 \times Q_1
\]
where the coordinate on $\AA^1$ is given by $1/h_0$. This is a closed
immersion, because where $h_0$ is nonzero, coordinates on $Q(V)$ with
can be recovered from the coordinates on the right-hand side
via \eqref{eq:zij}. We use this to construct the map
\[
\beta_1: B_1 \times Q_1=\Sh_U(B\times Q)/Q \times \AA^1 \to \Sh_U (B\times Q) \to B\times Q.
\]
The first arrow is given by the identity on the coordinates not in $Q(V)$,
while the coordinates on $Q(V)$ are computed via
\eqref{eq:zij}. The
second arrow projects into $B(V)\times Q(V)$. This map hits points in
$X(V)$ of the form $(v,\lambda v,\lambda,\mu (\lambda v)^2)$ with
$v,\lambda$ nonzero. 
\end{ex}

\subsection{Proof of Theorem~\ref{thm:parameterisation}}
\label{ssec:ProofParm}

\begin{proof}[Proof of Theorem ~\ref{thm:parameterisation}]
The $(\FIop)^k \times \Vec$-variety $Z$ is of product type, hence by
Definition~\ref{de:producttype} it can be written as
\[Z=[Y; B_1 \times Q_1, \dots, B_k \times Q_k]\] for some
$\Vec$-subvarieties $B_i$ of $Y\times \AA^{n_i}$ and pure
polynomial functors $Q_i$. Furthermore, $X$ is a proper
closed $(\FIop)^k \times \Vec$-subvariety of $Z$.

We prove, by induction on the quantity $\delta_X$, that all points in $X$
can be hit by partition morphisms from finitely many $(\FIop)^k \times
\Vec$-varieties $Z'$ of product type with $Z' \prec Z$. So in the proof
we may assume that this is true for all proper closed $(\FIop)^k \times
\Vec$-subvarieties $X' \subsetneq Z$ with $\delta_{X'}<\delta_X$.

Let $(S_1,\ldots,S_k) \in \FI^k$ be such that $\sum_i |S_i|=\delta_X$
and $X(S_1,\ldots,S_k) \neq Z(S_1,\ldots,S_k)$. If all $S_i$ are
empty, then set $Y'\coloneqq X(\emptyset,\ldots,\emptyset)$, a proper closed
$\Vec$-subvariety of $Y$, $B_i'\coloneqq Y' \times_Y B_i$, and $Z\coloneqq [Y';B_1' \times
Q_1,\ldots,B_k' \times Q_k]$. The partition morphism $(\id_{[k]},\phi)$
with $\phi(T_1,\ldots,T_k)$ the inclusion $\prod_i (B_i' \times
Q_i)^{T_i} \to \prod_i (B_i \times Q_i)^{T_i}$ has $X$ in its image,
and we have $Z' \prec Z$ because the $Q_i$ remain the same, $\LM(B_i')
\supseteq \LM(B_i)$ by Lemma~\ref{lm:BaseChangeLM}, and $Y'$ is a proper closed $\Vec$-subvariety of
$Y$. In this case, no shift of $Z$ is necessary.

Next assume that not all $S_i$ are empty. First we argue that the points
of $X(T_1,\ldots,T_k)$ where, for some $i$, $|T_i|$ is strictly smaller
than $|S_i|$, are hit by partition morphisms from finitely many
$Z' \prec Z$. We give the argument for $i=k$. Define the $k$-tuple
$S$ to be shifted over as $S\coloneqq (\emptyset, \dots, \emptyset, T_k) \in \FI^k$, and define the $(\FIop)^{k-1}\times \Vec$- variety $Z'$ of product type \[ Z' \coloneqq [(B_k \times Q_k)^{T_k}; B_1' \times Q_1, \dots, B_{k-1}'\times Q_{k-1}] \]
with $B_i' = (B_k \times Q_k)^{T_k} \times_Y B_i$.
Consider the partition morphism $(\pi, \phi): Z' \to \Sh_{S}Z$ where
$\pi: [k-1] \to [k]$ is the inclusion and
$\phi(T_1,\ldots,T_{k-1})$ is the natural isomorphism of
$\Vec$-varieties
\[ Z'(T_1,\ldots,T_{k-1}) \to
(\Sh_SZ)(T_1,\ldots,T_{k-1},\emptyset)=Z(T_1,\ldots,T_{k-1},T_k).
\]
Note that $\pi$ witnesses $Z' \preceq Z$ since the $Q_i$ with $i \leq
k-1$ remain the same and $\LM(B'_i) \supseteq \LM(B_i)$ by
Lemma~\ref{lm:BaseChangeLM}. Furthermore, since $k-1<k$, we have $Z'
\prec Z$ by Lemma~\ref{lm:Smallerlk}. All points in $X$
where the last index set has cardinality $|T_k|$ are hit by
this partition morphism. Since there are only finitely many
values of $|T_k|$ that are strictly smaller than $|S_k|$, we
are done.

So it remains to hit points in $X(T_1,\ldots,T_k)$ where
$|T_i| \geq |S_i|$ for all $i$. In this phase we will apply
Proposition~\ref{prop:Smaller}.

As by assumption not all $S_i$ are empty, after a
permutation of $[k]$ we may assume that $S_k \neq \emptyset$. Let $\ast$ be an element of $S_k$ and define $\widetilde{S_k}\coloneqq S_k \setminus \{\ast\}$.
Consider the $\Vec$-varieties
\begin{align*}Z(S_1, \dots, S_k) &= (B_1 \times Q_1)_Y^{S_1}
\times_Y \cdots \times_Y (B_k \times Q_k)_Y^{\widetilde{S_k}}
\times_Y (B_k \times Q_k)^{\{\ast\}} \text{ and}\\
\widetilde{Y}\coloneqq Z(S_1,\ldots,S_{k-1},\widetilde{S_k})&=
(B_1 \times Q_1)_Y^{S_1}
\times_Y \cdots \times_Y (B_k \times Q_k)_Y^{\widetilde{S_k}}.
\end{align*}
Set $\widetilde{B_k}\coloneqq  \widetilde{Y}\times_Y B_k \subseteq \widetilde{Y}\times
\AA^{n_k}$, and note that $X(S_1, \dots, S_k)$ is a proper
closed $\Vec$-subvariety of $\widetilde{B_k}\times Q_k$.  We
may therefore apply
Proposition~\ref{prop:Smaller} to $\widetilde{Y}, n_k,
\widetilde{B_k}, Q_k$ and
$X(S_1,\dots,S_k)$.

First consider the proper closed $\Vec$-subvariety $Y_0$ of
$\widetilde{Y}$ promised by Proposition~\ref{prop:Smaller}, and let
$X'$ be the largest closed $(\FIop)^k \times \Vec$-subvariety of $Z$
that intersects $Z(S_1,\ldots,S_{k-1},\widetilde{S_k})$ in $Y_0$. Then
$X'(S_1,\ldots,\widetilde{S_k}) \neq Z(S_1,\ldots,\widetilde{S_k})$,
and therefore $\delta_{X'} \leq \delta_{X}-1 < \delta_X$. Hence,
by the induction hypothesis, all points in $X'(T_1, \dots, T_k)$ can
be hit by finitely many partition morphisms from varieties $Z' \prec Z$
of product type.

Next we consider the remaining pieces of data from Proposition~\ref{prop:Smaller}.
First, we have the $\Vec$-variety $Y'$ with a morphism $\alpha: Y' \to \widetilde{Y}$.
Further, we have an integer $s \in \ZZ_{\geq 0}$ and for each $i = 0, \dots, s$ we have integers $n_{k+i}'$; $\Vec$-varieties $B_{k + i}' \subseteq Y' \times \AA^{n_{k+i}'}$; pure
polynomial functors $Q_{k+i}'$;
and morphisms $\beta_{k+i}: B_{k+i}'\times Q_{k+i}' \to \widetilde{B_k}\times Q_k$ satisfying the conditions (\ref{thesisone}) and (\ref{thesistwo}).

Define $B_i'\coloneqq Y' \times_{Y} B_i$ for $i = 1, \dots, k-1$ and the
$(\FIop)^{k+s}\times \Vec$-variety
\[ Z' \coloneqq [Y'; B_1'\times Q_1,
\dots, B_{k-1}'\times Q_{k-1}, B_{k}'\times Q'_{k}, \dots, B_{k+s}'\times
Q'_{k+s}].\]
Now the map $\pi:[k+s] \to [k]$ that is the identity on $[k-1]$ and maps
$[k+s] \setminus [k-1]$ to $\{k\}$ witnesses that $Z' \preceq Z$; here
we use that $B_{k+j}' \times Q'_{k+j} \prec B_k \times Q_k$ for $j \in
\{0,\ldots,s\}$ by the conclusion of Proposition~\ref{prop:Smaller}, and
also Lemma~\ref{lm:BaseChangeLM} to show that $B_i' \times Q_i \preceq
B_i \times Q_i$ for $i \in [k-1]$. In fact, we have $Z'
\prec Z$ by Lemma~\ref{lm:Smallerlk}.

Now the base variety $Y'$ of $Z'$ comes with a morphism
$\alpha$ to the base variety $\widetilde{Y}$ of $\Sh_S Z$;
we have morphisms $\beta_i: B_i' \times Q_i \to
\widetilde{B_i} \times
Q_i$ for $i=1,\ldots,k-1$ (the
natural map $B_i' \to \widetilde{B_i}$ times the identity on $Q_i$) and
the morphisms $\beta_{k+j}: B_{k+j}' \times Q_{k+j}' \to
\widetilde{B_k} \times Q_k$ defined earlier. By
Example~\ref{ex:PMProductType}, these data yield a partition
morphism $(\pi,\phi):Z' \to \Sh_S Z$. We have to show that this
partition morphism hits all points in $X$ that are not in
$X'$.

First we show, for a $V \in \Vec$, that a point $p \in
\Sh_{S}X(\widetilde{T_1}, \dots, \widetilde{T_{k}})(V)$ whose projection
to $\widetilde{Y}(V)$ is not in $Y_0(V)$ lies in the image
of $\phi(\widetilde{T_1},\ldots,\widetilde{T_k})(V)$.
To this end, we write
\[p = ((p_{i,t})_{t \in \widetilde{T_i}})_{i \in [k]}\] with
\[ p_{i,t}\in \Sh_{S}X(\emptyset, \dots, \emptyset, \{t\},
\emptyset, \dots, \emptyset)(V) = \widetilde{Y}(V)
\times_{Y(V)} B_i(V) \times Q_i(V) \subset
\widetilde{Y}(V)\times \AA^{n_i} \times Q_i(V)\]
where the singleton $\{t\}$ is in the $i$-th position.
We write $p_{i,t} = (\widetilde{y}, a_{i,t}, b_{i,t})$ with
$\widetilde{y}\in \widetilde{Y}(V)$, $a_{i,t} \in
\AA^{n_i}$, and $b_{i,t}\in Q_i(V)$.

By definition of a fibre product, the $p_{i,t}$ all have the same
projection $\widetilde{y}$ in $\widetilde{Y}(V)\setminus Y_0(V)$, and
hence we can apply (\ref{thesistwo}) of Proposition~\ref{prop:Smaller}
to the points $p_{k,t}$ with $t \in \widetilde{T_k}$. This yields integers
$l_t \in \{0,\ldots,s\}$ and points $p_{k,t}'
\in B_{k+l_t}'(V) \times Q_{k+l_t}'(V)$ for $t \in \widetilde{T_k}$ whose images in $Y'(V)$ are all equal, say to $y' \in Y'(V)$, and which
satisfy $\beta_{k+l_t}(V)(p_{k,t}')=p_{k,t}$ for all $t$. This implies
that $\alpha(y')=\widetilde{y}$.

Define
\[ T_{k+j}' \coloneqq \{\,t \in \widetilde{T_k}\; \mid
l_t=j\} \]
$j = 0,\dots, s$, and set $T_i'\coloneqq \widetilde{T_i}$ for $i = 1 ,\dots, k-1$.
In $Z'(T_1', \dots, T_{k+s}')$ we define the point $q = ((q_{i,t})_{t \in T_i'})_{i \in [k+s]}$ as follows.
We set $q_{i,t}$ to be $(y', a_{i,t}, b_{i,t})$ for $i = 1, \dots, k-1$ and $t \in T_i'$, and $q_{i,t} = p_{k,t}'$ for $i = k, \dots, k+s$ and $t \in T_i'$.
Then \[ \phi(T_1', \dots, T_{k+s}')(q) = p,\]
as desired.

Now, more generally, consider a point $p$ in $X(T_1, \dots, T_k)(V)
\setminus X'(T_1,\ldots,T_k)(V)$, where the cardinalities satisfy
$|T_i|\geq |S_i|$. Then there exists an $\FI^k$-morphism
$\iota = (\iota_1, \dots, \iota_k): S \to (T_1, \dots, T_k)$
such that $X(\iota)(p) \notin Y_0(V)$.
Define $\widetilde{T_i} \coloneqq T_i \setminus
\Im(\iota_i)$ and extend $\iota$ to an isomorphism $\iota^e:
S \sqcup (\widetilde{T_1}, \dots, \widetilde{T_k}) \to
(T_1,\ldots,T_k)$ by defining $\iota_i$ on $\widetilde{T_i}$
to be the inclusion.
Consider $X(\iota^e)(p) \in X(S \sqcup
(\widetilde{T_1}, \dots, \widetilde{T_k}))(V)$.
This is also a point in $\Sh_{S}X(\widetilde{T_1}, \dots,
\widetilde{T_k})(V)$ whose projection to $\widetilde{Y}(V)$ does
not lie in $Y_0(V)$.
We can therefore find a point $q$ as described above showing
that $X(\iota^e)(p)$ is in the image of $(\pi, \phi): Z' \to \Sh_{S}Z$;
by Definition~\ref{de:Image}, then so is $p$.
\end{proof}

\section{Proof of the Main Theorem} \label{sec:Proof}

The most general version of our Noetherianity result is the
following.

\begin{thm} \label{thm:Master}
Any $(\FIop)^k \times \Vec$-variety of product type is
Noetherian.
\end{thm}

\begin{proof}
We proceed by induction along the well-founded order on objects of
product type in $\bC$ from Section~\ref{sssec:ProductTypeOrder}.

Let $Z$ be an $(\FIop)^k \times \Vec$-variety of product type and
let $X_1 \supseteq X_2 \supseteq \ldots$ be a descending chain of
closed $(\FIop)^k \times \Vec$-subvarieties. Then either all $X_i$
are equal to $Z$, or there exists an $i_0$ such that $X\coloneqq X_{i_0}$
is a proper closed $(\FIop)^k \times \Vec$-subvariety of
$Z$. In the latter case, by
Theorem~\ref{thm:parameterisation}, there exist a finite number of
objects  $Z_1,\ldots,Z_N$ in $\bC$ of product type, along with $k$-tuples
$S_1,\ldots,S_N \in \FI^k$ and partition morphisms $(\pi_j,\phi_j):Z_j
\to \Sh_{S_j}Z$ such that every point of $X$ is hit by one of these. By
the induction hypothesis, all $Z_j$s are Noetherian. For each
$j$, by Lemma~\ref{lm:Preimage}, the
preimage in $Z_j$ of the chain $(\Sh_{S_j} X_{i})_{i \geq
i_0}$ is
a chain of closed subvarieties, which therefore stabilises. As soon as
these $N$ chains have all stabilised, then so has the chain
$(X_i)_i$---here we have used a version of Proposition~\ref{prop:coverNoetherian}.
\end{proof}

To deduce from this Theorems~\ref{thm:main} and~\ref{thm:Main2}, we
consider $\GL$-varieties $Z_1, \dots, Z_k$ as well as the product $Z\coloneqq Z_1^\NN \times \cdots \times Z_k^\NN$.
Recall Remark~\ref{re:Richer}.

\begin{proof}[Proof of Theorem~\ref{thm:Main2}]
We need to prove that any descending chain $Z \supseteq X_1 \supseteq
\ldots$ of $\Sym^k \times \GL$-stable closed subvarieties of $Z$
stabilises.

To each $Z_i$ is associated a $\Vec$-variety, which by abuse of notation
we also denote $Z_i$; see Remark~\ref{re:GLVar}. Furthermore, $Z_i$
is a closed subvariety of $B_i \times Q_i$ for some finite-dimensional
variety $B_i$ and some pure polynomial functor $Q_i$, and hence $Z$
is a closed subvariety of
\[ (B_1 \times Q_1)^\NN \times \cdots \times (B_k \times Q_k)^\NN. \]
Now each $X_i$ defines a closed $(\FIop)^k \times
\Vec$-subvariety $\widetilde{X_i}$ of
\[ \widetilde{Z}\coloneqq [Y;B_1 \times Q_1,\ldots,B_k \times Q_k],
\]
where $Y$ is a point. By Theorem~\ref{thm:Master}, the $\widetilde{X_i}$
stabilise. As soon as they do, so do the $X_i$.
\end{proof}

\begin{proof}[Proof of the Main Theorem]
Apply Theorem~\ref{thm:Main2} with $k = 1$.
\end{proof}

\printbibliography

\end{document}